\documentclass{amsart}

\IfFileExists{srcltx.sty}{\usepackage{srcltx}}

\numberwithin{equation}{section}

\usepackage[latin1]{inputenc}
\usepackage{xspace,amssymb,amsfonts,euscript}
\usepackage{amsthm,amsmath}
\usepackage{palatino}
\usepackage{euscript}
\input xy \xyoption {all}
\usepackage{enumerate}

\RequirePackage{color}
\definecolor{myred}{rgb}{0.75,0,0}
\definecolor{mygreen}{rgb}{0,0.5,0}
\definecolor{myblue}{rgb}{0,0,0.65}

\RequirePackage{ifpdf}
\ifpdf
  \IfFileExists{pdfsync.sty}{\RequirePackage{pdfsync}}{}
  \RequirePackage[pdftex,
   colorlinks = true,
   urlcolor = myblue, 
   citecolor = mygreen, 
   linkcolor = myred, 
   pagebackref,
   bookmarksopen=true]{hyperref}
\else
  \RequirePackage[hypertex]{hyperref}
\fi

\RequirePackage{ae, aecompl, aeguill} 

    \def\AM{{\mathbb{A}}}
  \def\bg{{\mathfrak b}}  
    \def\CM{{\mathbb{C}}}
    \def\DM{{\mathbb{D}}}
    \def\EM{{\mathbb{E}}}
    \def\FM{{\mathbb{F}}}
  \def\gg{{\mathfrak g}}  \def\GM{{\mathbb{G}}}

    \def\KM{{\mathbb{K}}}
  \def\lg{{\mathfrak l}}  \def\LM{{\mathbb{L}}}

    \def\OM{{\mathbb{O}}}
\def\PG{{\mathfrak P}}    \def\PM{{\mathbb{P}}}
    \def\QM{{\mathbb{Q}}}
    
\def\SG{{\mathfrak S}}  \def\sg{{\mathfrak s}}  
  \def\tg{{\mathfrak t}}  
  \def\ug{{\mathfrak u}}

\def\XG{{\mathfrak X}}    
    
    \def\ZM{{\mathbb{Z}}}

    \def\AC{{\mathcal{A}}}
    \def\BC{{\mathcal{B}}}
    \def\CC{{\mathcal{C}}}
    \def\DC{{\mathcal{D}}}
    
    \def\FC{{\mathcal{F}}}
    
    \def\HC{{\mathcal{H}}}

    \def\KC{{\mathcal{K}}}
    \def\LC{{\mathcal{L}}}
    \def\MC{{\mathcal{M}}}
    \def\NC{{\mathcal{N}}}
    \def\OC{{\mathcal{O}}}

    \def\RC{{\mathcal{R}}}
    \def\SC{{\mathcal{S}}}
    \def\TC{{\mathcal{T}}}

\def\a{\alpha}
\def\b{\beta}

\def\G{\Gamma}

\def\D{\Delta}
\def\e{\varepsilon}

\def\l{\lambda}

\def\s{\sigma}

\newcommand{\nc}{\newcommand} \newcommand{\renc}{\renewcommand}

\newcommand{\rdots}{\mathinner{ \mkern1mu\raise1pt\hbox{.}
    \mkern2mu\raise4pt\hbox{.}
    \mkern2mu\raise7pt\vbox{\kern7pt\hbox{.}}\mkern1mu}}

\def\mini{{\mathrm{min}}}

\def\pr{{\mathrm{pr}}}
\newcommand{\elem}[1]{\stackrel{#1}{\longto}}

\def\wh{\widehat}

\def\ov{\overline}
\def\un{\underline}

\def\p{{}^p}

\def\to{\rightarrow}

\def\longto{\longrightarrow}

\def\injto{\hookrightarrow}

\nc{\triright}{\stackrel{[1]}{\to}}
\nc{\longtriright}{\stackrel{[1]}{\longto}}

\nc{\Br}{\mathcal{B}}
\nc{\HotRR}{{}_R\mathcal{K}_R}
\nc{\HotR}{\mathcal{K}_R}
\nc{\excise}[1]{}
\nc{\defect}{\text{df}}
\nc{\h}[1]{\underline{H}_{#1}}

\nc{\Ga}{\mathbb{G}_a} 
\nc{\Gm}{\mathbb{G}_m} 

\nc{\Perv}{{\mathbf{P}}}

\nc{\IH}{{\mathrm{IH}}}

\nc{\gl}{{\mathfrak{gl}}}
\renc{\sl}{{\mathfrak{sl}}}
\renc{\sp}{{\mathfrak{sp}}}

\nc{\HBM}{H^{BM}}

\newtheorem{theorem}{Theorem}[section]
\newtheorem{lemma}[theorem]{Lemma}
\newtheorem{proposition}[theorem]{Proposition}
\newtheorem{corollary}[theorem]{Corollary}

\theoremstyle{definition}
\newtheorem{definition}[theorem]{Definition}

\theoremstyle{remark}

\def\op{{\mathrm{op}}}

\def\Xti{{\tilde{X}}}
\def\NCt{{\tilde{\mathcal{N}}}}

\def\p{{}^p}
\def\pp{{}^{p_+}\!}
\def\0{{}^0\!}

\DeclareMathOperator{\Ad}{{\mathrm{Ad}}}

\DeclareMathOperator{\codim}{codim}

\DeclareMathOperator{\End}{{\mathrm{End}}}

\DeclareMathOperator{\HH}{H}

\DeclareMathOperator{\ic}{{{\mathcal {IC}}}}

\DeclareMathOperator{\im}{{\mathrm{Im}}}

\DeclareMathOperator{\Irr}{{\mathrm{Irr}}}

\DeclareMathOperator{\nat}{{\mathrm{nat}}}

\DeclareMathOperator{\Rad}{{\mathrm{Rad}}}
\DeclareMathOperator{\res}{{\mathrm{res}}}

\DeclareMathOperator{\RHOM}{R\underline{Hom}}
\DeclareMathOperator{\Rg}{R\G}

\DeclareMathOperator{\Sh}{{\mathrm{Sh}}}
\DeclareMathOperator{\Soc}{Soc}
\DeclareMathOperator{\Spec}{{\mathrm{Spec}}}

\DeclareMathOperator{\Tr}{{\mathrm{Tr}}}
\DeclareMathOperator{\tr}{{\mathrm{tr}}}
\DeclareMathOperator{\Top}{Top}

\def\rs{{\mathrm{rs}}}

\def\mini{{\mathrm{min}}}

\def\triv{{\mathrm{triv}}}
\def\reg{{\mathrm{reg}}}

\def\subreg{{\mathrm{subreg}}}

\def\Sing{{\mathrm{Sing}}}

\def\boxempty{\square}

\def\impsi{\PG^0}

\def\Imp{\Longrightarrow}

\def\to{\rightarrow}

\def\longto{\longrightarrow}

\def\injto{\hookrightarrow}

\def\lamh{{\hat{\lambda}}}
\def\muh{{\hat{\mu}}}
\def\nuh{{\hat{\nu}}}
\def\etah{{\hat{\eta}}}
\def\zeth{{\hat{\zeta}}}
\def\NCh{{\hat{\mathcal{N}}}}

\begin{document}


\title{Modular Springer correspondence, decomposition matrices and basic sets}
\author{Daniel Juteau}
\address{LMNO, Universit\'e de Caen Basse-Normandie, CNRS, BP 5186, 14032 Caen, France}
\email{daniel.juteau@unicaen.fr}
\urladdr{http://www.math.unicaen.fr/~juteau}
\thanks{The author was supported by ANR Grants No.~ANR-09-JCJC-0102-01 and ANR-13-BS01-0001-01.}


\maketitle

\begin{abstract}
The Springer correspondence makes a link between the
characters of a Weyl group and the geometry of the
nilpotent cone of the corresponding semisimple Lie algebra.
In this article, we consider a modular version of the theory, and show that the
decomposition numbers of a Weyl group are particular cases of decomposition numbers for equivariant
perverse sheaves on the nilpotent cone. We give some decomposition
numbers which can be obtained geometrically. In the case of the symmetric group, we show that
James' row and column removal rule for the symmetric group can be
derived from a smooth equivalence between nilpotent singularities
proved by Kraft and Procesi. We give the complete structure of the Springer and Grothendieck sheaves
in the case of $SL_2$. Finally, we determine explicitly the modular Springer correspondence for
exceptional types.
\end{abstract}

\tableofcontents

\section{Introduction}
\label{intro}

\subsection{Modular Springer correspondence}

Let $G$ be a connected reductive group over a finite field $\FM_q$ of characteristic $p$,
with Frobenius endomorphism $F$.
In 1976, Springer defined a correspondence making a link between the
irreducible ordinary representations of the Weyl group $W$ of $G$ (over a field of
characteristic zero $\KM$) and the nilpotent orbits in the semi-simple
Lie algebra $\gg$ of $G$ \cite{SPRTRIG}, under some restrictions on $p$ and $q$. More precisely,
he constructed these representations in the top cohomology of some
varieties associated to the different nilpotent orbits, the Springer
fibers, which are the fibers of Springer's resolution of the nilpotent
cone. Each irreducible representation can be associated with a
nilpotent orbit and an irreducible $G$-equivariant local system on this orbit.
Springer's original approach was in terms of trigonometric sums, which
are (up to a scalar) Fourier transforms of characteristic functions of
$G^F$-orbits on $\gg^F$. This required to consider varieties over a
base field of characteristic $p$, and $\ell$-adic cohomology.

Then many other approaches to Springer correspondence were discovered.
For example, Kazhdan and Lusztig found a topological approach
\cite{KL}, and Slodowy constructed Springer representations by
monodromy \cite{SLO1}.
Links between different constructions were established, as in 
Hotta's work \cite{HOT}.
In the early 1980's, the emergence of intersection cohomology,
discovered by Goresky and MacPherson \cite{GM-IH1,GM-IH2}, and the
theory of perverse sheaves developed by Be\u{\i}linson, Bernstein,
Deligne and Gabber \cite{BBD}, allowed Lusztig, Borho and MacPherson
to reinterpret the correspondence with these new tools
\cite{LusGreen,BM}.
Later, Hotta and Kashiwara found an approach using a Fourier transform
for $\DC$-modules, when the base field is $\CM$ \cite{HK}, and
Brylinski adapted this idea to the case of a base field of
characteristic $p$, using a Fourier-Deligne transform for perverse
sheaves in $\ell$-adic cohomology \cite{BRY}.

Let us explain briefly the construction. Let $\KM$ be a finite extension of
$\QM_\ell$, with $\ell\neq p$. We call \emph{ordinary Springer sheaf}
the (derived) direct image of the constant perverse sheaf $\KM$
(suitably shifted) under the Springer resolution $\pi_\NC$ of the
nilpotent cone. By the Decomposition Theorem \cite{BBD} and the fact
that $\pi_\NC$ is semi-small, it is a semi-simple perverse sheaf.
Since it is also $G$-equivariant, its isotypic components are
parametrized by some pairs $(x,\rho)$, up to $G$-conjugation, where
$x$ is a nilpotent element and $\rho$ is an irreducible character of
the finite group $A_G(x) = C_G(x)/C_G(x)^0$. In both approaches using
perverse sheaves, one can show that the endomorphism algebra of $\KM
\KC_\NC$ is isomorphic to the group algebra $\KM W$. So the isotypic
components are also parametrized by the irreducible characters of $\KM
W$, and this defines the correspondence.

In the Lusztig-Borho-MacPherson approach, one uses a functor of
restriction to the nilpotent cone. Actually, one may give a direct
construction of the simple perverse sheaf associated
to an irreducible representation $\chi$ of $W$ \cite{LusGreen}. 
One takes the intersection cohomology complex of the local system on
the regular semisimple elements associated to $\chi$, and restricts it
to the nilpotent cone. Up to a shift by the rank of $G$, this is a simple
$G$-equivariant perverse sheaf on $\NC$, so it is associated with an
intersection cohomology datum consisting of a pair $(x,\rho)$ where
$x$ is a nilpotent element and $\rho$ is an irreducible representation
of $\KM A_G(x)$ (corresponding to an irreducible local system on the
orbit of $x$), up to $G$-conjugation, where $A_G(x)$ is the finite
group $C_G(x)/C_G(x)^0$ of components of the centralizer of $x$ in
$G$. This gives a correspondence which differs from Springer's original
parametrization by the sign character of $W$ \cite[Proposition 17.7]{SHOJI}. The proof that we get a
simple perverse sheaf on $\NC$ by this restriction process really
uses the Decomposition Theorem, and the result fails with characteristic
$\ell$ coefficients.

On the other hand, the Fourier-Deligne transform still makes sense with
characteristic $\ell$ coefficients, and sends a simple perverse sheaf to a simple
perverse sheaf. We can still associate
simple perverse sheaves to simple modules: if $\phi$ is a modular character of
$W$, one takes the Fourier transform of the intersection cohomology
complex of the local system on the regular semisimple elements
associated to $\phi$, which turns out to be supported by the nilpotent
cone, as in characteristic zero. For the proof, we
have to deal with non semi-simple perverse sheaves at some point, but
it still works.

\subsection{Decomposition matrices}
\label{intro:dec mat}

For a moment, suppose $W$ is any finite group, and choose a prime
number $\ell$. We also choose a sufficiently large finite extension
$\KM$ of $\QM_\ell$, with valuation ring $\OM$ and residue field
$\FM$. In representation theory, a triple such as $(\KM,\OM,\FM)$ is called
an $\ell$-modular system. If $E$ is a simple $\KM W$-module, we may
choose a $W$-stable $\OM$-lattice $E_\OM$ in $E$, and form the $\FM
W$-module $\FM \otimes_\OM E_\OM$. Though its isomorphism class
depends on the choice of the lattice, it does have a well-defined
class in the Grothendieck group $K_0(\FM W)$. Thus, for each simple
$\FM W$-module $F$, we have a well-defined multiplicity
$d^W_{E,F} := [\FM \otimes_\OM E_\OM : F]$, which is called a
decomposition number. The matrix
\[
D^W := (d^W_{E,F})_{E\in \Irr \KM W,\ F \in \Irr \FM W}
\]
is called the ($\ell$-modular) decomposition matrix of $W$. We refer
the reader to \cite[Partie III]{Serre} for more information.

When the ordinary characters are known, the determination of $D^W$ is
equivalent to the determination of the modular characters.
Modular representations of finite groups are much less known than
ordinary ones. For example, for the symmetric group $\SG_n$, we know
the ordinary characters, but in characteristic $\ell$, even if we know
how to parametrize the simple modules \cite{JAMESirr}, we do not even know their
dimensions explicitly in general.

When $W$ is a Weyl group (as before), we will see that the
determination of $D^W$ can be translated into a geometrical or
topological problem, thanks to the Springer correspondence.
Indeed, one can define a decomposition matrix $D^\NC$ for the
$G$-equivariant perverse sheaves on the nilpotent cone, just
as in representation theory \cite{decperv}, and we will see that $D^W$
is a submatrix of $D^\NC$. Thus to determine $D^W$, it would be enough
to determine explicitly the modular Springer correspondence, and to
compute the stalks of the modular intersection cohomology
complexes. While the first task has now been achieved by \cite{JLS, AHJRII} and the present paper,
the second one is probably extremely difficult, however one can
hope to prove some qualitative results with these methods.

\subsection{Some explicit results}

The nilpotent orbits are naturally ordered by the inclusion of their
closures, the smallest orbit being the trivial orbit $\OC_\triv =
\{0\}$, and the largest orbit being the regular orbit $\OC_\reg$.
Suppose $G$ is simple. Then there is a unique minimal non-trivial
orbit $\OC_\mini$, and a unique maximal non-regular orbit, the
subregular orbit $\OC_\subreg$. For $G$ of adjoint type, we computed
the decomposition numbers corresponding to the regular and subregular
orbits, and to the minimal and trivial orbits in \cite{cohmin,decperv}. 
These results will be recalled.

A result of Kraft and Procesi \cite{KP3} implies that the special
pieces of the nilpotent cone are $\FM_\ell$-smooth as soon as $\ell
\neq 2$, in classical types. We will deduce that some decomposition
numbers, involving a special orbit and a smaller orbit in the same
piece, are zero for $\ell \neq 2$.

Then we focus on type $A$. We determine explicitly the modular
Springer correspondence in this case, and we show that James' row and
column removal rule \cite{JAMESdecIII} is a consequence of a smooth
equivalence of nilpotent singularities obtained by Kraft and Procesi
\cite{KP1}.

In the case of $G = SL_2$, we give the complete structure of the Springer
and Grothendieck sheaves.

Finally, we determine explicitly the modular Springer correspondence for exceptional
Weyl groups, using their decomposition matrices.

\subsection{Related work}
This paper contains the results of \cite[Chapters 6 and 8]{Juteau:thesis} (finally!),
the results in Subsection \ref{subsec:basicsets} which were obtained at MSRI in 2008 during
the program ``Representation Theory of Finite Groups and Related Topics'',
and the explicit results for exceptional Weyl groups.

Several papers related to modular Springer theory have appeared in the
mean time. In \cite{Treumann:springer}, Rossmann's topological approach is followed to produce
some version of the induction theorem for modular Springer representations.

In \cite{Mautner:schur}, Mautner proves that, for $n \geq
d$, the category of $GL_d$-equivariant perverse sheaves with $\EM$
coefficients on the nilpotent cone of $\gg\lg_d$ is
equivalent to the category of polynomial representations of $GL_n$
over $\EM$ of degree $d$, using Lusztig's embedding of the nilpotent
cone in the affine Grassmannian \cite{LusGreen}, a map in the other
direction at the level of stacks, and the geometric Satake
correspondence \cite{MV}. Using the modular Springer correspondence,
he then provides a geometric proof of Schur-Weyl duality, the Schur
functor being described by homomorphisms from the Springer
sheaf. 

In \cite{Achar-Mautner:ringel}, Achar and Mautner study the functor
from the equivariant derived category of the nilpotent cone to itself
given by Fourier transform on $\gg$ followed by restriction to
$\NC$. They prove that it is an equivalence. In the case of
$GL_n$ it is a geometric version of Ringel duality, exchanging tilting
sheaves (which are parity sheaves \cite{JMW}) and projective
sheaves. Moreover, they investigate examples in types $B_2$ and $G_2$.

There is now a modular generalized Springer correspondence
\cite{AHJRI, AHJRII, AHJRIII}. It is always defined (except possibly
for type $E_8$ in characteristic $2$); we can classify cuspidal pairs completely for classical
types, and partially for exceptional types; we can determine the generalized correspondence explicitly for $SL_n$
in any characteristic, for classical types in characteristic $2$, and partially in exceptional types.
In \cite{JLS}, we determine the non-generalized modular Springer correspondence for classical groups
in odd characteristic. This last work uses basic sets for several orders (the Springer order,
and some combinatorial orders which are known to be compatible with decomposition numbers), which
justifies the introduction of the notion of ``basic set datum'', where the order is part of the data.

The fact that the Fourier transform is an auto-equivalence makes it easy to show that the endomorphism algebra
of the Springer sheaf is again the group algebra of the Weyl group.
See \ref{subsubsec:res} for a discussion of the approach by restriction to the nilpotent cone (either using
the fact that both constructions differ by the sign character \cite{AHJR}, or directly with the Borel-Moore homology of the Steinberg variety \cite{Riche:res}).

We should also mention \cite{AHR} which, apart from being a continuation of \cite{AH}, contains
many compatibilities, notably of induction and restriction functors with respect to a modular
Springer functor defined in terms of restriction to the nilpotent cone. By \cite{AHJR},
those compatibilities are valid also for the Fourier transform construction of the present paper.

Finally, the need to better understand the geometry of nilpotent cones to do further explicit calculation of stalks of modular
IC sheaves was one of the motivations for \cite{FJLS}. We describe generic singularities of nilpotent orbit closures
for exceptional types (for classical types, this was done in \cite{KP1,KP2}). Most of the time, we still find
simple and minimal singularities, extending the applicability of the results in Section \ref{sec:dec}.

\subsection{Outline}

In Section \ref{sec:preliminaries}, we recall some facts used in the sequel, mostly about modular perverse sheaves.
In Section \ref{sec:fourier}, we recall the definition and properties
of the Fourier-Deligne transform that we need.
In Section \ref{sec:springer}, we recall the background for Springer
correspondence and we explain why we can still define such a
correspondence in the modular case, using the Fourier-Deligne transform.
In Section \ref{sec:decomposition}, we show that the decomposition
matrix of a Weyl group is part of the decomposition matrix for
equivariant perverse sheaves on the nilpotent cone, and that this defines a ``Springer basic set'' for the Weyl group.
In Section \ref{sec:dec}, we give some decomposition numbers which can
be obtained geometrically, using previous work.
Section \ref{sec:gl_n} is devoted to the study of modular Springer correspondence for the general linear group,
including the row and column removal rule.
In Section \ref{sec:sl2}, we give explicit calculations
of the structure of the Springer and Grothendieck sheaves for $G = SL_2$.
Finally, in Section \ref{sec:exc}, we determine explicitly the modular Springer correspondence
by displaying the decomposition matrices of the exceptional Weyl groups in the order given by Springer correspondence.

\subsection*{Acknowledgements}
I sincerely thank my former supervisors Cédric Bonnafé and Raphaël Rouquier for giving me such a nice subject.
Besides making great contributions to modular Springer theory, my collaborators
Pramod Achar, Anthony Henderson, Cédric Lecouvey, Simon Riche and Karine Sorlin 
are to be praised for their (big) patience while I was not submitting this paper.
I am very grateful to Hyohe Miyachi for recomputing the decomposition matrices of the exceptional Weyl groups in GAP3.
Thanks are also due to Jean Michel for his help with GAP3.
Finally, I thank Olivier Dudas for his careful reading and his friendly encouragements.

\section{Preliminaries}
\label{sec:preliminaries}

Let $k$ be a finite field $\FM_q$ or its algebraic closure.
We will consider only $k$-varieties (separated $k$-schemes of
finite type), and morphisms of $k$-schemes. We call $p$ the
characteristic of $k$.

Recall from \ref{intro:dec mat} that $(\KM,\OM,\FM)$ is an
$\ell$-modular system. We assume $\ell$ is different from $p$.
Let $X$ be a $k$-variety. For $\EM = \KM, \OM, \FM$, we have an
abelian category of constructible $\EM$-sheaves $\Sh(X,\EM)$. It
contains the full subcategory of $\EM$-local systems. For $X$
connected, the latter is equivalent to the category of continuous
representations of the fundamental group of $X$. We denote by
$D^b_c(X,\EM)$ the bounded constructible derived category of
$\EM$-sheaves as defined by Deligne. In this setting we have
Grothendieck's six operations: the tensor product $- \otimes_\EM^L -$,
the internal Hom functor $\RHOM(-,-)$, and, for $f : X \to Y$, the
pairs of adjoint functors $(f^*,f_*)$ and $(f_!,f^!)$ (all of these
denote functors between derived categories). We denote by $\DM_X$ the
Grothendieck-Verdier duality functor.
We also have functors of extension of scalars
$\KM(-) := \KM \otimes_\OM (-)$ and of modular reduction 
$\FM(-) := \FM \otimes^L_\OM (-)$.

The category $\p\MC(X,\EM)$ of $\EM$-perverse sheaves on $X$ is the
heart of the perverse $t$-structure on $D^b_c(X,\EM)$ for the middle
perversity \cite{BBD}. In the case where $\EM$ is a field, this
$t$-structure is stable by the duality, but for $\EM = \OM$ the
duality exchanges this standard perverse $t$-structure  with a dual
perverse $t$-structure with heart $\pp\MC(X,\EM)$. Over a point,
the dual perverse sheaves are complexes of $\OM$-modules with torsion-free $H^0$
and torsion $H^1$ (both finitely generated), all other cohomology
groups being zero.
We have perverse cohomology functors
$\p H^m:D^b_c(X,\EM)\to\p\MC(X,\EM)$, and also
$\pp H^m:D^b_c(X,\EM)\to\pp\MC(X,\EM)$ in case $\EM = \OM$.
Distinguished triangles give rise to long exact sequences in perverse
cohomology. 
We refer to \cite[\S  3.3]{BBD} and \cite{decperv}.

\subsection{About intermediate extensions}
\label{subsec:IC}

Assume we are in a recollement situation \cite[\S 1.4]{BBD}.
So we have three $t$-categories (i.e. triangulated
categories endowed with $t$-structures) $\DC$,
$\DC_U$ and $\DC_F$ related by six gluing functors, as follows:
\begin{equation}\label{eq:recollement setup}
\xymatrix{
\DC_F
\ar[r]^{i_*}
&
\DC
\ar@<-3ex>[l]_{i^*}
\ar@<3ex>[l]_{i^!}
\ar[r]^{j^*}
&
\DC_U
\ar@<-3ex>[l]_{j_!}
\ar@<3ex>[l]_{j_*}
}
\end{equation}
satisfying certain axioms.
The $t$-structure on $\DC$ is deduced from the $t$-structures on
$\DC_U$ and $\DC_F$ by the recollement procedure. The hearts are
denoted by $\CC$, $\CC_U$ and $\CC_F$.
If $T$ is any of the six gluing functors, we denote by $\p T$
the corresponding functor between the hearts, obtained as the
composition of the inclusion of the heart of the source $t$-category,
followed by $T$, and then by the perverse cohomology functor $\p H^0$
of the target $t$-category. Recall that $i_*$ and $j^*$ are $t$-exact
and give rise to exact functors between the hearts ($i_*$ is the
inclusion of a thick subcategory, while $j^*$ is the corresponding
quotient functor).

For example $\p i_* \p i^!$ is the functor
``largest subobject in $\p i_*\CC_F$'', while  $\p i_* \p i^*$ is the
functor ``largest quotient in $\p i_*\CC_F$''. Note that we write $\p i_*$
to indicate that we consider the functor between the hearts, although it is
common to drop the $p$ from the notation because the functor $i_*$ is
already $t$-exact. We will occasionally make this abuse of notation. The
same remark applies to the functor $\p j^*$.

It follows from the axioms that there is a canonical morphism
$j_! \to j_*$, and the latter factors
through a canonical morphism $\p j_! \to \p j_*$. The intermediate
extension functor $\p j_{!*}$ is defined as the image of this
morphism. Over $\OM$, there are also versions for $p_+$ of all those
functors.

If $\SC$ (resp. $\SC_U$, $\SC_F$) denotes the set of (isomorphisms
classes of) simple objects in $\CC$ (resp. $\CC_U$, $\CC_F$),
then we have $\SC = \p j_{!*} \SC_U \cup \p i_* \SC_F$.

The typical recollement situation arises when one
considers an open immersion $j:U\to X$ with closed complement 
$i:F\to X$. The perverse $t$-structure on a stratified space is
defined inductively by gluing shifts of natural $t$-structures on each
stratum.

One bad feature of the intermediate extension functor $\p j_{!*}$ is
that it is not exact. It may already happen with perverse sheaves in
characteristic zero, however it is easy not to be aware of this
problem because one usually applies this functor to a semisimple
object. But with perverse sheaves with $\FM$ coefficients, the problem
arises very often, and already in the Springer correspondence for
$GL_2$ in characteristic $2$.
Nevertheless, the functor of intermediate extension enjoys some nice
properties, that we will now recall. We refer to \cite{decperv} for
more details, proofs and references.

\begin{proposition}
The intermediate extension functor preserves monomorphisms and
epimorphisms.
\end{proposition}
Using the preceding result one can show the following.
\begin{proposition}
\label{prop:top socle}
Let $A$ be an object of $\CC_U$. Then we have
\begin{gather*}
\Soc \p j_{!*} A
\simeq \Soc \p j_* A
\simeq \p j_{!*} \Soc A,\\
\Top \p j_{!*} A
\simeq \Top \p j_! A
\simeq \p j_{!*} \Top A.
\end{gather*}
\end{proposition}

Thus the intermediate extension functor preserves tops and socles. The
second series of results is about composition multiplicities.

\begin{proposition}
If $B$ is a finite length object in $\CC$, then $\p j^*B$ is of finite
length in $\CC_U$ and we have
\[
[B : \p j_{!*} S] = [\p j^* B : S]
\]
for all simple objects $S$ in $\CC_U$.
\end{proposition}

Since the three functors $\p j_!$, $\p j_{!*}$ and $\p j_*$ are
extension functors (meaning that applying $\p j^*$ to them gives back the
identity), we obtain the following corollary.

\begin{corollary}
\label{cor:IC mult}
If $A$ is an object in $\CC_U$, then we have
\[
[\p j_! A : \p j_{!*} S] = [\p j_{!*}A : \p j_{!*}S]
= [\p j_* A : \p j_{!*} S] = [A:S]
\]
for all simple objects $S$ in $\CC_U$,
whenever those multiplicities are defined (i.e. when the objects
involved have finite length).
\end{corollary}

So the intermediate extension functor preserves composition
multiplicities. This will imply that, in the context of perverse
sheaves, the intermediate extension functor preserves decomposition
numbers. Let us first recall from \cite{decperv} that, in an
open/closed situation, starting with a torsion-free object $A$ in 
$\p \MC(U,\OM)$, we can define a priori nine extensions of $\FM A$
with the functors $\FM\, \p j_?$, $\p j_? \FM$ and $\FM\, \pp j_?$
where $? \in \{!, !*, *\}$. Out of those, $\p j_!$ and $\pp j_*$ might
fail to be perverse but the other seven are automatically perverse.
The three involving $!*$ ones are particularly interesting as they are the
modular reduction of the $\ic$, the $\ic$ of the modular reduction,
and the modular reduction of the dual $\ic$. Since they are all
extensions of $\FM A$, using the proposition we get the following.

\begin{corollary}
\label{cor:open mult}
If $A$ is a torsion-free object in $\p \MC(X,\OM)$, then
\[
[\FM\, \p j_{!*}A : \p j_{!*}S] 
= [\p j_{!*} \FM A : \p j_{!*}S] 
= [\FM\, \pp j_{!*}A : \p j_{!*}S] 
= [\FM A:S].
\]
\end{corollary}

Now let $X$ be a stratified algebraic variety.
Recall from \cite{decperv} that we can define decomposition numbers
for perverse sheaves by
\[
d^X_{(\OC,\LC),(\OC',\LC')} := [\FM \ic(\ov\OC,\LC_\OM) : \ic(\ov{\OC'},\LC')]
\]
where $\OC$ and $\OC'$ are two strata, $\LC$ is a $\KM$-local system
on $\OC$ with some integral form $\LC_\OM$, and $\LC'$ is an
$\FM$-local system on $\OC'$. For simplicity, we assume that the
representations of the fundamental groups corresponding to the
considered local systems factor through a finite group, so that it is
clear that a stable lattice exists and that decomposition numbers are
well defined --- this will be the case for the nilpotent cone
stratified by the group orbits.
(However, we note that a finite dimensional continuous representation
over $\KM$ of a profinite
group like the étale fundamental group always has a stable lattice.)
The preceding corollary implies that decomposition numbers between
$\ic$'s with the same support are just decomposition numbers for (a
finite quotient of) the fundamental group of the stratum,
as stated below.

\begin{corollary}
\label{cor:dec same}
Let $\OC$ be a stratum in $X$ and let $\LC$ and $\LC'$ be two local
systems on $\OC$ corresponding to irreducible representations $\rho$ and $\rho'$
of $\pi_1(\OC)$ which factor through a finite quotient $H$. Then
\[
d^X_{(\OC,\LC),(\OC,\LC')} = d^H_{\rho,\rho'}.
\]
\end{corollary}

We have a natural partial order on the strata of $X$ defined by
$\OC \geq \OC'$ if $\ov \OC \supset \OC'$.
For decomposition numbers involving different strata, one can make the
following easy observation.

\begin{proposition}
\label{prop:dec tri}
If some decomposition number $d^X_{(\OC,\LC),(\OC',\LC')}$ is nonzero,
then $\OC' \leq \OC$.
\end{proposition}

\begin{proof}
If $\LC_\OM$ is an integral form of the local system $\LC$, then
$\ic(\ov\OC,\LC_\OM)$ is an integral form of $\ic(\ov\OC,\LC)$, and is
clearly supported by $\ov\OC$, hence $\FM\ic(\ov\OC,\LC)$ is also
supported by $\ov\OC$. Now $\p i_* \p \MC(\ov\OC, \FM)$, where
$i:\ov\OC\to X$ is the closed inclusion, is a Serre
subcategory of $\p \MC(X, \FM)$, hence the result.
\end{proof}

Let us assume that $X$ is a $G$-variety with finitely many orbits, and
let us consider only $G$-equivariant local systems and perverse
sheaves. The $G$-equivariant local systems on some orbit $\OC$
correspond to finite dimensional representations of the finite group
$A_G(\OC) = C_G(x) / C_G^0(x)$ where $x$ is some point in $\OC$.
Corollary \ref{cor:dec same} and Proposition \ref{prop:dec tri} mean
that the decomposition matrix for $G$-equivariant perverse sheaves on
$X$ is block triangular, and the diagonal blocks are decomposition
matrices for the groups $A_G(\OC)$.

We have used the fact that an $\ic$ complex over $\KM$ has some
integral form, namely an $\ic$ complex over $\OM$, with the same support.
In the next subsection we will see a better statement: for any
perverse sheaf on some closed subvariety, any integral form is
supported by that closed subset.

\subsection{Integral forms and supports}

In this paragraph we prove a result which we will need in Section
\ref{sec:springer}, and which would have found its place in
\cite[\S 2.5]{decperv}. We consider a $k$-variety $X$ with an open
subvariety $j : U \to X$ and closed complement $i : Z \to X$. So we
are in a recollement situation.

\begin{proposition}
\label{prop:int supp F}
Suppose $A$ is a $\KM$-perverse sheaf on $X$ supported on $Z$, and let
$A_\OM$ be an integral form of $A$, that is, a torsion-free
$\OM$-perverse sheaf on $X$ such that $\KM A_\OM \simeq A$.
Then $A_\OM$ is also supported by $Z$.
\end{proposition}

\begin{proof}
First, we have $\KM\, j^* A_\OM = j^* A = 0$, hence
$j^* A_\OM$ is torsion. It follows that
$\p j_! j^* A_\OM$ is torsion as well.
Now, from the distinguished triangle
\[
j_! j^* A_\OM \longto A_\OM \longto i_* i^* A_\OM \rightsquigarrow
\]
we get the following perverse cohomology exact sequence:
\[
\p j_! j^* A_\OM \longto A_\OM \longto \p i_* \p i^* A_\OM \longto 0
\]
since $j_!$ is right $t$-exact. Now, the first map must be zero as
it goes from a torsion object to a torsion-free object.
Thus $A_\OM \simeq \p i_* \p i^* A_\OM$ is supported by $Z$.
\end{proof}

\subsection{Small and semi-small morphisms}

Let us recall the notions of semi-small and small morphisms, and a
proposition which shows their usefulness (the usual proof applies).

\begin{definition}
A morphism $\pi : \Xti \to X$ is \emph{semi-small} if there is a stratification
$\XG$ of $X$ such that the for all strata $S$ in $\XG$, and for all closed
points $s$ in $S$, we have $\dim \pi^{-1}(s) \leqslant \frac{1}{2}\codim_X(S)$.
If moreover these inequalities are strict for all strata of positive codimension,
we say that $\pi$ is \emph{small}.
\end{definition}

\begin{proposition}
\label{prop:small}
Let $\pi : \Xti \to X$ be a surjective, proper and separable morphism,
with $\Xti$ smooth irreducible of dimension $d$.
Let $\LC$ be an $\EM$-local system on $\Xti$. Let us consider the complex $K = \pi_*\LC[d]$.
\begin{enumerate}[(i)]
\item If $\pi$ is semi-small, then $\pi$ is generically finite and $K$ is $p$-perverse.
\item If $\pi$ is small, then $K = \p j_{!*}\p j^* K$
for any inclusion $j : U \to X$ of a smooth open dense subvariety over which
$\pi$ is étale.
\end{enumerate}
\end{proposition}

\subsection{Rational smoothness, $\ell$-smoothness}

Suppose $X$ is an irreducible variety. If $j:V\to X$ the
inclusion of a smooth open dense subvariety and $\LC$ is a local
system on $V$, then we denote by $\ic(X,\LC)$ the intermediate
extension $\p j_{!*}(\LC[\dim X])$.
We say that $X$ is $\EM$-smooth if $\ic(X,\EM)$ is reduced to
$\un \EM_X[\dim X]$. When $\EM = \OM$, we require this condition for both
perversities, $p$ and $p_+$. This property ensures that $X$ satisfies
Poincaré duality with $\EM$ coefficients.

\begin{proposition}
\label{prop:E smooth}
Let $H$ be a finite group of order prime to $\ell$.
If $X$ is an $\FM_\ell$-smooth $H$-variety, then $X/H$ is also
$\FM_\ell$-smooth.
\end{proposition}

\section{Fourier-Deligne transform}
\label{sec:fourier}

We will recall well known facts about the Fourier-Deligne transform
(see \cite{Lau}), pointing out that it makes sense with coefficients
in $\KM$, $\OM$ or $\FM$. It is a pleasant exercise to fill in the details
in Laumon's elegant exposition. One may refer to
\cite[Chapter 5]{Juteau:thesis} for details.

In this section, $k = \FM_q$.
Let us assume that $\EM^\times$ contains a primitive root of unity of order $p$.
We fix a non-trivial character $\psi : \FM_p \to \EM^\times$, that is,
a primitive root $\psi(1)$ of order $p$ in $\EM^\times$.
Composing with $\Tr_{\FM_q/\FM_p}$, we get a character of $\FM_q$.
Let $\LC_\psi$ be the locally constant $\EM$-sheaf of rank $1$ 
on the additive group $\GM_{a}$ associated to $\psi$ (the corresponding Artin-Schreier
local system).

Let $S$ be a variety, and $E\elem{\pi} S$
a vector bundle of constant rank $r \geqslant 1$.
We denote by $E' \elem{\pi'} S$ its dual vector bundle,
by $\mu : E \times_S E' \to \AM^1$ the canonical pairing,
and by $\pr : E \times_S E' \to E$ and $\pr' : E \times_S E' \to E'$
the canonical projections. So we have the following diagram.

\[
\xymatrix{
& E\times_S E'
\ar[dl]_\pr
\ar[dr]^{\pr'}
\ar[rr]^\mu 
&& \AM^1\\
E \ar[dr]_\pi 
\ar@{}[rr] | {\boxempty}
&& E' \ar[dl]^{\pi'}\\
&S
}
\]

\begin{definition}
The Fourier-Deligne transform for $E \elem{\pi} S$, associated to the
character $\psi$, is the triangulated functor
\[
\FC_\psi : D^b_c(E,\EM) \longto D^b_c(E',\EM)
\]
defined by
\[
\FC_\psi(K) = {\pr'}_! (\pr^* K \otimes^\LM_\EM \mu^*\LC_\psi) [r]
\]
\end{definition}

In the sequel, we will drop the indices $\psi$ from the notations
$\FC_\psi$ and $\LC_\psi$ when no confusion may arise.

Let $E''\elem{\pi''}S$ be the bidual vector bundle of $E\elem{\pi}S$
and $a : E \elem{\sim}E''$ the $S$-isomorphism defined by $a(e) = -\mu(e,-)$
(that is, the opposite of the canonical $S$-isomorphism). We will denote by
$\s:S\to E$, $\s':S\to E'$ and $\s'':S\to E''$ the respective null sections
of $\pi$, $\pi'$ and $\pi''$. Finally, we will denote by
$s:E\times_S E\to E$ the addition of the vector bundle $E\elem{\pi}S$
and by $- 1_E:E\to E$ the opposite for this addition.

The following Proposition is the analogue of the fact that the Fourier transform
of the constant function is a Dirac distribution supported at the origin in
classical Fourier analysis. By the function/sheaf dictionary, this becomes
a functorial isomorphism, to which we will refer as \eqref{eq:dirac}.
It will be used to prove that the pair corresponding to the trivial
character $1_W$ is the trivial pair $(0,1)$ consisting of the zero
nilpotent orbit with the trivial local system, both in the ordinary and
modular cases.

\begin{proposition}
\label{prop:dirac}
We have a functorial isomorphism 
\begin{equation*}
\label{eq:dirac}
\tag{DIRAC}
\FC(\pi^* L[r]) \simeq \s'_* L (-r)
\end{equation*}
for all objects $L$ in $D^b_c(S,\EM)$.
\end{proposition}

The following fundamental result says that there is an inverse Fourier
transform. We will use it to see that the Springer correspondence is injective.

\begin{theorem}
\label{th:inv}
Let $\FC'$ be the Fourier-Deligne transform, associated to the character
$\psi$, of the vector bundle $E'\elem{\pi'}S$. Then we have a functorial
isomorphism
\begin{equation*}
\label{eq:inv}
\tag{INV}
\FC' \circ \FC (K) \simeq a_* K (-r)
\end{equation*}
for all objects $K$ in $D^b_c(E,\EM)$.
\end{theorem}

\begin{corollary}
\label{cor:equiv db}
The triangulated functor $\FC$ is an equivalence
of triangulated categories from
$D^b_c(E,\EM)$ to $D^b_c(E',\EM)$, with quasi-inverse
$a^* \FC'(-)(r)$.
\end{corollary}

We will also need the following result, which describes how the
Fourier-Deligne transform behaves with respect to morphisms of vector
bundles.

\begin{theorem}
\label{th:morphism}
Let $f : E_1 \to E_2$ be a morphism of vector bundles over $S$,
with constant ranks $r_1$ and $r_2$ respectively, and
let $f' : E'_2 \to E'_1$ denote the transposed morphism.
Then we have a functorial isomorphism
\begin{equation*}
\label{eq:morphism}
\tag{MOR}
\FC_2(f_! K_1) \simeq f'^* \FC_1(K_1)[r_2 - r_1]
\end{equation*}
for $K_1$ in $D^b_c(E_1,\EM)$, where $\FC_1$ and $\FC_2$ denote the
Fourier-Deligne transforms for $E_1$ and $E_2$.
\end{theorem}

With this result, one can generalize Proposition \ref{prop:dirac}
\eqref{eq:dirac} to any sub-vector bundle:

\begin{proposition}
\label{prop:sub}
Let $i : F \hookrightarrow E$ be a sub-vector bundle over $S$, with constant
rank $r_F$. We denote by $i^\perp : F^\perp \hookrightarrow E'$ the orthogonal
of $F$ in $E'$. Then we have a canonical isomorphism
\begin{equation*}
\label{eq:sub}
\tag{SUB}
\FC(i_* \EM_F [r_F]) \simeq i^\perp_* \EM_{F^\perp}(-r_F)[r - r_F]
\end{equation*}
\end{proposition}

Now we consider how the Fourier-Deligne transform behaves with respect
to a base change. We will need \eqref{eq:bc_!} to prove that the
Fourier-Deligne transform of the Grothendieck sheaf is the Springer sheaf.

\begin{proposition}
\label{prop:bc}
Let $f : S_1 \to S$ be an $\FM_q$-morphism of finite type. We have a
base change diagram:
{\scriptsize
\[
\xymatrix{
&
E_1 \times_{S_1} E'_1
\ar[dl]^{\pr_1}
\ar[dr]_{\pr'_1}
\ar[drrr]^F
\ar[rrrrr]^{\mu_1}
&&&&&
\AM^1
\\
E_1
\ar[dr]_{\pi_1}
\ar[drrr]^(.3){f_E}
&&
E'_1
\ar[dl]_(.3){\pi'_1} |!{[ll];[dr]}\hole
\ar[drrr]^(.3){f_{E'}} |!{[rr];[dr]}\hole
&&
E \times_S E'
\ar[urr]_\mu
\ar[dl]^(.7)\pr
\ar[dr]^{\pr'}
\ar@{}[ul] | {\D}
\\
&
S_1
\ar[drrr]_f
&& E \ar[dr]_\pi
&& E' \ar[dl]^{\pi'}
\\
&&&& S
}
\]
}
Let $\FC_1$ denotes the Fourier-Deligne transform associated to
$E_1 \elem{\pi_1} S$.
Then we have functorial isomorphisms
\begin{equation*}
\label{eq:bc^*}
\tag{$\text{BC}^*$}
\FC_1 (f_E^* K) \simeq f_{E'}^* \FC(K)
\end{equation*}
\vspace{-.7cm}
\begin{equation*}
\label{eq:bc_!}
\tag{$\text{BC}_!$}
\FC({f_E}_!\ K_1) \simeq {f_{E'}}_!\ \FC_1(K_1)
\end{equation*}
for $K$ in $D^b_c(E,\EM)$ and $K_1$ in $D^b_c(E_1,\EM)$.
\end{proposition}

The next result shows that the Fourier-Deligne transform
preserves equivariance.

\begin{proposition}
\label{prop:G eq}
Let $G$ be a smooth affine group scheme over $S$, acting linearly on
the vector bundle $E \elem{\pi} S$, let $K$ and $L$ be two objects in
$D^b_c(E,\EM)$, and let $M$ be an object in $D^b_c(G,\EM)$. We denote
by $m : G \times_S E \to E$ the action of $G$ on $E$, and by $m' : G
\times_S E' \to E'$ the contragredient action, defined by $m'(g,e') =
{}^t g^{-1}.e'$. Then each isomorphism
\[
m^* K \simeq M \boxtimes^\LM_S L
\]
in $D^b_c(G\times_S E, \EM)$ induces canonically an isomorphism
\begin{equation*}
\label{eq:G eq}
\tag{$G$-EQ}
m'^*\FC(K) \simeq M \boxtimes^\LM_S \FC(L)
\end{equation*}
in $D^b_c(G\times_S E', \EM)$.
\end{proposition}

One could have used $\pr'_!$ instead of $\pr'_*$ to define an a priori
different notion of Fourier-Deligne transform. The following
fundamental result shows that both are canonically isomorphic
\cite[Appendice 2.4]{KaLa}.

\begin{theorem}
\label{th:oubli support}
For any object $K$ in $D^b_c(E,\EM)$, the support forgetting morphism
\begin{equation*}
\label{eq:oubli support}
\tag{SUPP}
\pr'_!\ (\pr^* K \otimes^\LM_\EM \mu^* \LC) \longto \pr'_*\ (\pr^* K \otimes^\LM_\EM \mu^* \LC)
\end{equation*}
is an isomorphism.
\end{theorem}

This theorem implies the powerful property that Fourier transform
commutes with duality.

\begin{theorem}
\label{th:rhom fourier}
We have a functorial isomorphism
\[
\RHOM(\FC_\psi(K), \pi'^! L) \simeq \FC_{\psi^{-1}}(\RHOM(K,\pi^! L)) (r)
\]
for $(K,L)$ in $D^b_c(E,\EM)^\op \times D^b_c(S,\EM)$.
\end{theorem}

Remember that, if $X$ is a variety, we denote by $\DC_{X,\EM}$ the
duality functor of $D^b_c(X,\EM)$.  If
$a : X \to \Spec k$ is the structural morphism, we denote by
$D_{X,\EM}$ the dualizing complex $a^! \EM$.

\begin{corollary}
\label{cor:dual fourier}
We have a functorial isomorphism
\[
\DC_{E',\EM} (\FC_\psi(K)) \simeq \FC_{\psi^{-1}} (\DC_{E,\EM}(K)) (r)
\]
for $K$ in $D^b_c(E,\EM)^\op$.
\end{corollary}

\begin{theorem}
\label{th:equiv perv}
$\FC$ maps ${}^p\MC(E,\EM)$ onto ${}^p\MC(E',\EM)$. The functor
\begin{equation*}
\label{eq:equiv perv}
\tag{EQUIV}
\FC : {}^p\MC(E,\EM) \longto {}^p\MC(E',\EM)
\end{equation*}
is an equivalence of abelian categories, with quasi-inverse $a^* \FC'(-)(r)$.
\end{theorem}

Moreover, by Proposition \ref{prop:G eq} \eqref{eq:G eq}, $\FC$
sends $G$-equivariant perverse sheaves on $G$-equivariant perverse
sheaves (take the constant perverse sheaf on $G$ for $M$).

\begin{corollary}
\label{cor:fourier simple}
Suppose $\EM = \KM$ or $\FM$. Then $\FC$ transforms simple
$\EM$-perverse sheaves on $E$ into simple $\EM$-perverse sheaves on
$E'$.
\end{corollary}

This will play a crucial role in the construction of Springer correspondence.
Again, we have a version with $G$-equivariant perverse sheaves.

\section{Springer correspondence}
\label{sec:springer}

\subsection{The geometric context}\label{subsec:geometric context}

\subsubsection{Notation}

Let $G$ be a split connected semisimple linear algebraic group of rank $r$
over $k$. Let us fix a Borel
subgroup $B$ of $G$, with unipotent radical $U$, and a maximal torus
$T$ contained in $B$. We denote by $\gg$, $\bg$, $\ug$ and $\tg$ the
corresponding Lie algebras. The characters of $T$ form a free abelian
group $X(T)$ of rank $r$.  The Weyl group $W = N_G(T) / T$ acts as a
reflection group on $V = \QM \otimes_\ZM X(T)$.

Let $\Phi \subset X(T)$ be the root system of $(G,T)$, $\Phi^+$ the
set of positive roots defined by $B$, and $\D$ the corresponding
basis. We denote by $\nu_G$ (or just $\nu$) the cardinality of
$\Phi^+$. Then $\dim G = 2\nu +r$, $\dim B = \nu + r$, $\dim T = r$
and $\dim U = \nu$.

\subsubsection{The finite quotient map}

Let $\phi : \tg \to \tg/W$ be the quotient map, corresponding to the
inclusion $k[\tg]^W \hookrightarrow k[\tg]$. It is finite and
surjective. For $t \in \tg$, we will also denote $\phi(t)$ by $\ov t$.

Let us assume that $p$ is not a torsion prime for $\gg$.  Then
$k[\tg]^W = k[\phi_1,\ldots,\phi_r]$ for some algebraically
independent homogeneous polynomials $\phi_1,\ldots,\phi_r$ whose degrees
$d_1 \leqslant \ldots \leqslant d_r$ are well defined and called
\emph{characteristic degrees} \cite[Théorème 3]{DEM}.
We have $d_i = m_i + 1$, where the $m_i$ are the exponents of $W$
(determining the eigenvalues of a Coxeter element) \cite[Chap. VI, \S 6.2]{BOUR456}.
The quotient space $\tg/W$ can be identified with $\AM^r$ and $\phi$ with
$(\phi_1,\ldots,\phi_r)$.

For example, if $G = SL_n$, we can identify $\tg$ with the hyperplane
$\{ (x_1,\ldots,x_n) \mid x_1 + \cdots + x_n = 0 \}$ of $k^n$, and for
$\phi_i$ ($1 \leq i \leq r = n - 1$) we can take the $(i+1)$st
elementary symmetric polynomial $\sigma_{i+1}$ on $k^n$, restricted to
this hyperplane ($\s_1$ does not appear, since its restriction vanishes).

\subsubsection{The adjoint quotient}

Assume that $p > 2$ if $G$ has a component of type $C_m$, $m\geq 1$.
Then by \cite{CR} the Chevalley restriction theorem holds: 
the restriction map $k[\gg]^G \to k[\tg]^W$ is an isomorphism.
Recall that we assume moreover that $p$ is not a torsion prime for
$G$, so that the preceding paragraph applies. Let $\chi_i$, $1\leq i
\leq r$, be the element of $k[\gg]^G$ restricting to $\phi_i$.
Then the $\chi_i$ are homogeneous algebraically independent polynomials
of degrees $d_1 \leqslant \ldots \leqslant d_r$.
Hence we have a morphism $\chi = (\chi_1,\ldots,\chi_r) : \gg \to
\gg/\!/G \simeq \tg/W \simeq \AM^r$.  It is called the Steinberg map,
or the adjoint quotient. 

The morphism $\chi$ has been extensively studied (see \cite{SLO2} and the
references therein).  First, it is flat, and its schematic fibers are
irreducible, reduced and normal complete intersections, of codimension
$r$ in $\gg$.  If $t \in \tg$, let $\gg_{\ov t}$ be the fiber
$\chi^{-1}(\ov t)$.  It is the union of finitely many adjoint orbits. It
contains exactly one orbit of regular elements, which is open and
dense in $\gg_{\ov t}$, and whose complement has codimension
$\geqslant 2$ in $\gg_{\ov t}$.  This regular orbit is exactly the
smooth locus of $\gg_{\ov t}$.  So $\tg/W$ parametrizes
the orbits of regular elements.  The fiber $\gg_{\ov t}$ also
contains exactly one orbit of semisimple elements, the orbit of $t$,
which is the only closed orbit in $\gg_{\ov t}$, and which lies in the
closure of every other orbit in $\gg_{\ov t}$.
In fact, $\chi$ can be interpreted as the map which sends $x$ to the
intersection of the orbit of $x_s$ with $\tg$, which is a $W$-orbit.

For example, for $G = SL_n$, we can define the $\chi_i : \sg\lg_n \to
k$ by the formula
\[
\det(\xi - x) = \xi^n + \sum_{i = 1}^{n - 1} (-1)^{i + 1} \chi_i(x) \xi^{n - 1 - i} \in k[\xi]
\]
for $x \in \sg\lg_n$. So $\chi(x)$ can be interpreted as the
characteristic polynomial of $x$. Restricting $\chi_i$ to $\tg$, we
recover $\phi_i = \sigma_{i+1}$.

\subsubsection{Springer's resolution of the nilpotent cone}

Let $\NC$ be the closed subvariety of $\gg$ consisting of its
nilpotent elements.  It is the fiber $\gg_{0} = \chi^{-1}(0)$. In
particular, it is a complete intersection in $\gg$, given by the
equations $\chi_1(x) = \cdots = \chi_r(x) = 0$.  It is singular. We
are going to describe Springer's resolution of the nilpotent cone.

The set $\BC$ of Borel subalgebras of $\gg$ is a homogeneous space
under $G$, in bijection with $G/B$, since the normalizer of $\bg$ in
$G$ is $B$. Hence $\BC$ is endowed with a structure of smooth
projective variety, of dimension $\nu$.

Let $\NCt = \{ (x,\bg') \in \NC \times \BC \mid x \in \bg' \}$.
It is a smooth variety of dimension $2\nu = \dim\NC$: the second projection makes it
a vector bundle over $\BC$, which can be identified with $G\times^B \ug$
via $g *^B x \mapsto (\Ad(g)(x), \Ad(g)(\bg))$. It can be further identified with the
cotangent bundle $T^*\BC$, since $T\BC = T(G/B) = G \times^B (\gg/\bg)$
and $\ug = \bg^\perp$.  Now let $\pi_\NC:\NCt\to\NC$ be the first
projection. Since $\NCt$ is closed in $\NC \times \BC$ and $\BC$ is
projective, the morphism $\pi_\NC$ is projective. Moreover, it is an
isomorphism over the open dense subvariety of $\NC$ consisting of the
regular nilpotent elements, so $\pi_\NC$ is indeed a resolution of
$\NC$.

\subsubsection{Grothendieck's simultaneous resolution of the adjoint quotient}

In the last paragraph, we have seen the resolution of the fiber
$\chi^{-1}(0)$.  We are now going to explain Grothendieck's
simultaneous resolution, which gives resolutions for all the fibers of
$\chi$ simultaneously.

Let $\tilde\gg = \{ (x,\bg') \in \gg \times \BC \mid x \in \bg' \}$
and let $\pi:\tilde\gg\to\gg$ be the first projection. With the second projection,
$\tilde\gg$ is a vector bundle over $\BC$, isomorphic to
$G \times^B \bg$ via $g *^B x \mapsto (\Ad(g)(x), \Ad(g)(\bg))$.
Hence $\tilde\gg$ is a smooth variety of dimension $2\nu + r = \dim \gg$.
Finally, we define $\theta$ as the composition
$\tilde\gg \simeq G\times^B \bg \to \bg/[\bg,\bg]\stackrel{\sim}{\to} \tg$.
Then the commutative diagram
\[
\xymatrix{
\tilde \gg
\ar[r]^\pi
\ar[d]_\theta
& \gg \ar[d]^\chi
\\
\tg \ar[r]_\phi
& \tg/W
}
\]
is a simultaneous resolution of the
singularities of the flat morphism $\chi$.  That is, $\theta$ is
smooth, $\phi$ is finite surjective, $\pi$ is proper, and $\pi$
induces a resolution of singularities $\theta^{-1}(t) \to
\chi^{-1}(\phi(t))$ for all $t \in \tg$.

\subsection{Springer correspondence for $\EM W$}
\label{subsec:springer}

If $\EM = \KM$ or $\FM$, let $\PG_\EM$ denote the set of pairs
$(\OC,\LC)$, where $\OC$ is some nilpotent orbit and $\LC$ is an
irreducible $G$-equivariant local system on $\OC$. Those pairs index
the simple $G$-equivariant perverse sheaves on $\NC$. So their classes
form a basis of the Grothendieck group of the category of
$G$-equivariant perverse sheaves on $\NC$. In formulas, we have
\[
K_0(\p \MC(\NC,\EM)) \simeq \bigoplus_\OC K_0(\EM A_G(\OC)) \simeq \bigoplus_{(\OC,\LC) \in \PG_\EM} \ZM [\ic(\ov\OC,\LC)].
\]

The aim of this subsection is to define a Springer correspondence for
$\EM W$, that is an injection $\Psi_\EM : \Irr \EM W \to \PG_\EM$, valid in
both cases $\EM = \KM$ and $\EM = \FM$. 

\subsubsection{The perverse sheaves $\KC_\rs$, $\KC$ and $\KC_\NC$}

Let us consider the following commutative diagram with cartesian squares:

\[
\xymatrix{
\tilde\gg_\rs
\ar[d]_{\pi_\rs}
\ar@<-0.5ex>@{^{(}->}[r]^{\tilde j_\rs}
\ar@{}[dr] | {{\boxempty_\rs}}
&
\tilde\gg
\ar[d]^\pi
\ar@{}[dr] | {{\boxempty_\NC}}
&
\NCt
\ar@<0.5ex>@{_{(}->}[l]_{i_\NCt}
\ar[d]^{\pi_\NC}
\\
\gg_\rs
\ar@<-0.5ex>@{^{(}->}[r]_{j_\rs}
&
\gg
&
\NC
\ar@<0.5ex>@{_{(}->}[l]^{i_\NC}
}
\]

Let us define the complex
\[
\KC := \pi_* \OM_{\tilde \gg} [2\nu + r].
\]

Note that $\pi_* = \pi_!$ since $\pi$ is proper. Let $\EM$ be $\KM$, $\OM$ or $\FM$.
Since modular reduction commutes with direct images,
we have $\EM \KC = \pi_* \EM_{\tilde \gg} [2\nu + r]$.
By the proper base change theorem, the fiber at a point $x$ in $\gg$ of $\EM \KC$
is given by $(\EM \KC)_x = \Rg(\BC_x, \EM)$.

Let $\KC_\rs = j_\rs^*\KC$ and $\KC_\NC = i_\NC^*\KC[-r]$.
By the proper base change theorem and the commutation between modular reduction
and inverse images, we have
\begin{gather*}
\EM\KC_\rs = j_\rs^*\EM \KC = {\pi_\rs}_* \EM_{{\tilde \gg}_\rs} [2\nu + r],\\
\EM\KC_\NC = i_\NC^*\EM \KC[-r] = {\pi_\NC}_* \EM_\NCt [2\nu].
\end{gather*}

The morphism $\pi$ is separable, proper and small, hence
$\EM\KC$ is an intersection cohomology complex by Proposition \ref{prop:small}.
The morphism $\pi_\rs$ obtained after the base change $j_\rs$
is a Galois finite \'etale covering, with Galois group $W$. In
particular, $W$ is a quotient of the étale fundamental group of
$\gg_\rs$. If $E$ is a representation of $W$ (over $\KM$, $\OM$ or
$\FM$), we will simply denote by $\un E$ the associated local system
on $\gg_\rs$, shifted by the dimension $2\nu + r$.
Then we have $\EM\KC = {j_\rs}_{!*} \EM\KC_\rs = {j_\rs}_{!*} \un{\EM W}$.
Note that, if $\EM = \OM$, we have
${}^{p_+} {j_\rs}_{!*} \KC_\rs
= \DC_{\gg,\OM} ({}^p {j_\rs}_{!*} \KC_\rs)
= \DC_{\gg,\OM} (\KC) = \KC$
so it does not matter whether we use $p$ or $p_+$ (we have used the fact that
the regular representation is self-dual, and that $\KC$ is self-dual because
$\pi$ is proper and $\tilde\gg$ is smooth).

Thus the endomorphism algebra of $\EM\KC_\rs$ is the group algebra
$\EM W$. Since the functor ${j_\rs}_{!*}$ is fully faithful, it induces an isomorphism
$\End(\KC_\rs) = \EM W \elem{\sim} \End(\KC)$.
In particular, we have an action of $\EM W$ on the stalks
$\HC^i_x(\EM\KC) = H^{i + 2 \nu + r}(\BC_x,\EM)$. 

When $\EM = \KM$, the group algebra $\KM W$ is semisimple, and so are the perverse sheaves
\[
\KM\KC_\rs \simeq \bigoplus_{E \in \Irr \KM W} \un E^{\dim E}
\]
and
\[
\KM\KC \simeq \bigoplus_{E \in \Irr \KM W} ({j_\rs}_{!*} \un E)^{\dim E}.
\]

If $\ell$ does not divide the order of the Weyl group $W$,
then we have a similar decomposition for $\EM = \OM$ or $\FM$.
However, we are mostly interested in the case where $\ell$ divides $|W|$. Then
$\FM\KC_\rs$ and $\FM\KC$ are not semisimple.
More precisely, we have decompositions
\begin{gather*}
\OM W = \bigoplus_{F \in \Irr \FM W} P_F^{\dim F}\\
\FM W = \bigoplus_{F \in \Irr \FM W} (\FM P_F)^{\dim F}
\end{gather*}
where $P_F$ is a projective indecomposable $\OM W$-module such that
$\FM P_F$ is a projective cover of $F$.
Besides, $\FM P_F$ has top and socle isomorphic to $F$.

Hence we have a similar decomposition for $\KC_\rs$, and its modular reduction:
\begin{gather*}
\KC_\rs = \bigoplus_{F \in \Irr \FM W} \un {P_F}^{\dim F},\\
\FM \KC_\rs = \bigoplus_{F \in \Irr \FM W} (\un {\FM P_F})^{\dim F}.
\end{gather*}
These are decompositions into indecomposable summands, and the
indecomposable summand $\un{\FM P_F}$ has top and socle
isomorphic to $\un F$.  By Proposition \ref{prop:top socle},
applying ${j_\rs}_{!*}$ we get decompositions into indecomposable
summands, and the indecomposable summand ${j_\rs}_{!*} (\un{\FM P_F})$ 
has top and socle isomorphic to ${j_\rs}_{!*} \un F$.
\begin{gather*}
\KC = \bigoplus_{F \in \Irr \FM W} ({j_\rs}_{!*} \un{P_F})^{\dim F}\\
\FM \KC = \bigoplus_{F \in \Irr \FM W} ({j_\rs}_{!*} \un{\FM P_F})^{\dim F}
\end{gather*}

The morphism $\pi_\NC$ is proper and semi-small, hence $\EM\KC_\NC$
is perverse. The functor ${i_\NC}^*(-)[-r]$ induces a morphism
\begin{equation}\label{mor:res}
\res : \End(\EM\KC) \longto \End(\EM\KC_\NC)
\end{equation}

\subsubsection{Springer correspondence by restriction}
\label{subsubsec:res}

\begin{theorem}
\label{th:res}
The morphism $\res$ in (\ref{mor:res}) is an isomorphism.
\end{theorem}

In the $\EM = \KM$ case this was proved by Borho and MacPherson
in \cite{BM}. They argued as follows.
The Weyl group $W$ acts on $G/T$ by $gT \cdot w = g n_w T$, where $n_w$ is
any representative of $w$ in $N_G(T)$. So $W$ acts naturally on the
cohomology complex $\Rg(G/T,\EM)$. Since the projection $G/T \to G/B$
is a locally trivial $U$-fibration, it induces an isomorphism
$\Rg(G/B,\EM) \simeq \Rg(G/T,\EM)$, and thus there is a natural action
of $W$ on $\Rg(G/B,\EM)$, and hence on the cohomology $\HH^*(G/B,\EM)$.
Let us call it the classical action. On the other hand, the stalk at
$0$ of $\KC$ is isomorphic to $\Rg(G/B,\EM)$, and 
$\EM W$ acts on it through $\EM W \simeq \End(\KC) \elem{\res}
\End(\KC_\NC)$. Let us call it the Lusztig action \cite{LusGreen}.
In fact, the two actions coincide. 
When $\EM = \KM$, one can show that the classical
action on the cohomology is the regular representation.  
Since the regular representation is faithful,
this implies that the morphism $\res$ is injective.
Then one can show that the two algebras have the same dimension
using results from \cite{STEIN2}. This proves that $\res$ is an
isomorphism in the case $\EM = \KM$.

We have
\[
\KM\KC = \bigoplus_{E \in \Irr \KM W} {j_\rs}_{!*} \un E^{\dim E}
\]
and $i_\NC^*(\KM\KC) [-r] = \KM\KC_\NC$. In fact, the restriction functor
$i_\NC^* [-r]$ sends each simple constituent ${j_\rs}_{!*} \un E$
on a simple object. The assignment
\[
\RC : E \mapsto i_\NC^* {j_\rs}_{!*} \un E [-r]
\]
is an injective map from $\Irr \KM W$ to the simple $G$-equivariant
perverse sheaves on $\NC$, which are parametrized by the pairs
$(\OC,\LC)$, where $\OC$ is a nilpotent orbit, and $\LC$ is an
irreducible $G$-equivariant $\KM$-local system on $\OC$. This is the
Springer correspondence (by restriction).

For example, for $G = SL_n$, the Specht module $S^\l$ is sent to
$\ic(\ov\OC_\l,\KM)$, where $\OC_\l$ is the nilpotent orbit
corresponding to the partition $\l$ by the Jordan normal form.

When $\EM = \OM$ or $\FM$, several difficulties arise with the
Borho-MacPherson approach to prove Theorem \ref{th:res}. First, if we
consider the cohomology of the stalk at $0$, then one loses
information, and obtains a non-faithful representation of the Weyl
group, already in the case of $G = SL_2$. It is likely that the map
from $\EM$ to the endomorphism algebra of the stalk
at $0$ considered as a (perfect) \emph{complex} of $\EM W$-modules is
injective. This would be one way to prove the injectivity, but this
complex does not seem so easy to understand. Also, working over $\OM$,
one cannot invoke a dimension argument to conclude that we have an
isomorphism.

To avoid these problems, we will use another classical approach to the
Springer theory, using a Fourier transform. This gives a different
action of the Weyl group on the Springer sheaf. Actually both actions
differ by the sign character of the Weyl group, as was proved by Hotta
in \cite{HOT} with $\KM$ coefficients, and in \cite{AHJR}
with general $\EM$ coefficients. So we can deduce Theorem \ref{th:res}
from the corresponding result about the Fourier approach.

We note that Riche has been able to prove Theorem \ref{th:res}
directly, using the Ginzburg interpretation of the endomorphism
algebra of the Springer sheaf in terms of the Borel-Moore homology of
the Steinberg variety \cite{Riche:res}.

\subsubsection{The Fourier-Deligne transform of $\EM\KC$}

We assume that there exists a non-degenerate $G$-invariant symmetric
bilinear form $\mu$ on $\gg$, so that we can identify $\gg$ with its
dual. This is the case, for example, if $p$ is very good for $G$ (take
the Killing form), or if $G = GL_n$ (take $\mu(X,Y) = \tr(XY)$).  For
a more detailed discussion, see \cite{LET}.

\begin{lemma}
The root subspace $\gg_\a$ is orthogonal to $\tg$ and to all the root subspaces
$\gg_\b$ with $\b \neq -\a$.
\end{lemma}

\begin{proof}
Let $x \in \tg$. For $t \in T$, we have
$\mu(x,e_\a) = \mu(\Ad(t)x, \Ad(t)e_\a) = \a(t) \mu(x, e_\a)$.
Since $\a \neq 0$, we can choose $t$ so that $\a(t) \neq 1$, and thus
$\mu(x, e_\a) = 0$.

Now let $\b$ be a root different from $-\a$. We have
$\mu(e_\b, e_\a) = \mu(\Ad(t)e_\b, \Ad(t)e_\a) = \a(t)\b(t)\mu(e_\b,e_\a)$.
Since $\b \neq -\a$, we may choose $t$ so that $\a(t)\b(t) \neq 1$, and thus
$\mu(e_\b, e_\a) = 0$.
\end{proof}

\begin{corollary}
\label{b orthogonal}
The orthogonal of $\bg$ in $\gg$ is $\ug$.
\end{corollary}

\begin{proof}
By the preceding lemma, $\bg$ is orthogonal to $\ug$, and we have
$\dim \bg + \dim \ug = 2\nu + r = \dim \gg$, hence the result, since
$\mu$ is non-degenerate. 
\end{proof}

Let $\FC$ be the Fourier-Deligne transform associated to $p : \gg \to
S = \Spec k$
(any vector space can be considered as a vector bundle over a point).
Since we identify $\gg$ with $\gg'$, the functor $\FC$
is an auto-equivalence of the triangulated category $D^b_c(\gg,\EM)$.
The application $a$ of Theorem \ref{th:inv} \eqref{eq:inv},
which was defined as the opposite of the canonical
isomorphism from a vector bundle to its bidual, is now multiplication by $-1$.

We will need to consider the base change $f : \BC \to \Spec k$. We will denote by
$\FC_\BC$ the Fourier-Deligne transform associated to $p_\BC : \BC \times \gg \to \BC$.
We have a commutative diagram
\[
\xymatrix{
G \times_B \gg
\ar@{=}[d]
&
G \times_B \bg
\ar@{=}[d]
\ar@<.5ex>@{_{(}->}[l]
&
G \times_B \ug
\ar@{=}[d]
\ar@<.5ex>@{_{(}->}[l]
\\
\BC \times \gg
\ar[d]_{p_\BC}
\ar[dr]_F
&
\tilde\gg
\ar@<.5ex>@{_{(}->}[l]_i
\ar[d]^\pi
\ar@{}[dl] |(.3){\D}
\ar@{}[dr] | {\boxempty_\NC}
&
\NCt
\ar@<.5ex>@{_{(}->}[l]_{i_\NCt}
\ar[d]^{\pi_\NC}
\\
\BC
\ar[dr]_f
&
\gg
\ar[d]^p
&
\NC
\ar@<.5ex>@{_{(}->}[l]_{i_\NC}
\\
&
\Spec k
}
\]

We have (below COM means isomorphisms coming from commutative diagrams)
\[
\begin{array}{rcll}
\FC(\EM\KC)
&=& \FC(\pi_!\ \EM_{\tilde\gg}[2\nu + r])
\\
&=& \FC(F_!\ i_*\ \EM_{\tilde\gg}[\nu + r]) [\nu]
&\text{by COM}_!(\D)
\\
&=& F_!\ \FC_\BC(i_*\ \EM_{\tilde\gg}[\nu + r]) [\nu]
&\text{by BC}_!(f)
\\
&=& F_!\ i_*\ {i_\NCt}_*\ \EM_\NCt (-\nu - r) [\nu] [\nu]
&\text{by SUB}
\\
&=& {i_\NC}_*\ {\pi_\NC}_!\ \EM_\NCt (-\nu - r) [2\nu]
&\text{by COM}_!(\D,\boxempty_\NC)
\\
&=& {i_\NC}_*\ \EM\KC_\NC (- \nu - r)
\end{array}
\]

Applying $\FC$ and using Theorem \ref{th:inv} \eqref{eq:inv}, we get
\[
a_* \EM\KC (-2\nu - r) = \FC({i_\NC}_*\ \EM\KC_\NC) (-\nu - r)
\]
But $a_* \EM\KC \simeq \EM\KC$ since $\EM\KC$ is monodromic
($\CM^*$-equivariant), so that:
\begin{theorem}
We have
\begin{gather*}
\FC(\EM\KC) \simeq {i_\NC}_*\ \EM\KC_\NC (- \nu - r)\\
\FC({i_\NC}_*\ \EM\KC_\NC) \simeq \EM\KC (-\nu)
\end{gather*}
\end{theorem}

Note that this proves a second time that $\KC_\NC$ is perverse.

\begin{corollary}
The functors ${j_\rs}_{!*}$, $\FC(-)(\nu + r)$ and ${i_\NC}_*$ induce isomorphisms
\[
\EM W
= \End(\EM \KC_\rs)
\elem{\sim} \End(\EM \KC)
\elem{\sim} \End({i_\NC}_*\ \EM\KC_\NC)
\stackrel{\sim}{\longleftarrow} \End(\EM\KC_\NC)
\]
\end{corollary}

For $E \in \EM W\text{-mod}$, let
$\TC(E) = \FC {j_\rs}_{!*}(E\ [2\nu + r])(\nu + r)$.
By the theorem above, we have $\TC(\EM W) = \KC_\NC$. More correctly,
we should say that $\TC(\EM W)$ is supported on $\NC$, and write
$\TC(\EM W) = {i_\NC}_* \KC_\NC$, but we identify the perverse sheaves
on $\NC$ with their extension by zero on $\gg$.

\begin{corollary}
\label{cor:dec K}
The perverse sheaf $\KM \KC_\NC$ is semisimple, and we have the decomposition
\[
\KM \KC_\NC =
\bigoplus_{E \in \Irr \KM W} \TC(E)^{\dim E}.
\]
Similarly, we have decompositions into indecomposable summands
\[
\KC_\NC = 
\bigoplus_{F \in \Irr \FM W} \TC(P_F)^{\dim F},
\]
and
\[
\FM \KC_\NC
= \bigoplus_{F \in \Irr \FM W} \TC(\FM P_F)^{\dim F}.
\]
The indecomposable summand $\TC(\FM P_F)$ has top and socle isomorphic
to $\TC(F)$.
\end{corollary}

\subsubsection{Springer correspondence by Fourier-Deligne transform.}

We are now ready to define a Springer correspondence by Fourier-Deligne transform. We
treat the ordinary and modular cases simultaneously. The ordinary case
was achieved by Brylinski \cite{BRY}, following 
the case of $\DC$-modules over complex varieties \cite{HK}. See the
survey by Shoji \cite{SHOJI} for a nice exposition.

At the time of \cite{Juteau:thesis}, the modular case was new. All the steps go through as in the ordinary
case, except the fact that $\FM\KC_\NC$ is no longer semisimple in
general. However, to each simple $\FM W$-module, we can still
associate a simple $G$-equivariant perverse sheaf on the nilpotent
cone. Recall that we denote by $\PG_\EM$ the set of pairs $(\OC,\LC)$
where $\OC$ is a nilpotent orbit and $\LC$ is a $G$-equivariant
irreducible local system on $\OC$ with coefficients in $\EM = \KM$ or
$\FM$. We will often abuse notation and identify it with the set of pairs $(x,\rho)$
with $x\in \NC$ and $\rho \in \Irr \EM A_G(x)$, up to $G$-conjugacy.
For $(x,\rho) \in \PG_\EM$, we denote by $\ic(x,\rho)$ the
corresponding simple $G$-equivariant perverse sheaf.

\begin{theorem}
\label{th:springer}
Let $\EM$ be $\KM$ or $\FM$. For $E \in \Irr \EM W$, let
\[
\TC_\EM(E) = \FC j_{\rs !*} (\un E)
\]
Then $\TC_\EM(E)$ is a simple $G$-equivariant perverse sheaf supported by $\NC$.
Hence it is of the form $\ic_\EM(x_E,\rho_E)$ for some
$(x_E,\rho_E)\in \PG_\EM$. The assignment $E \mapsto (x_E,\rho_E)$ defines
an injective map
\[
\Psi_\EM : \Irr \EM W \longto \PG_\EM.
\]
\end{theorem}

\begin{proof}
Since $E$ is an irreducible $\EM W$-module, the shifted local system
$\un E$ is a simple perverse sheaf on $\gg_\rs$, so its intermediate
extension $j_{\rs !*} (\un E)$ is a simple perverse sheaf on
$\gg$. Since the Fourier-Deligne transform maps a simple perverse
sheaf to a simple perverse sheaf by Corollary \ref{cor:fourier
  simple}, and preserves $G$-equivariance by Proposition \ref{prop:G eq},
we conclude that $\TC_\EM(E)$ is a simple $G$-equivariant perverse sheaf.

Let us now prove that $\TC_\EM(E)$ is supported by $\NC$. In the case
$\EM = \KM$, one can simply say that $\TC_\KM(E)$ is a direct summand
of $\KM\KC_\NC$ by Corollary \ref{cor:dec K}. In the case $\EM = \FM$,
this is no longer true when $\ell$ divides the order of the Weyl group
$W$. However, we have seen in Corollary \ref{cor:dec K} that
$\TC_\FM(E)$ is isomorphic to the top and to the socle of $\TC_\FM(\FM
P_E)$, which is a direct summand of $\FM\KC_\NC$. Thus $\TC_\FM(E)$ is
supported by $\NC$ in the modular case as well.

It follows that $\TC_\EM(E)$ is of the form $\ic_\EM(x_E,\rho_E)$ for some
uniquely determined $(x_E,\rho_E) =: \Psi_\EM(E)$ in $\PG_\EM$.  The
map $\Psi_\EM$ is injective, because, if $\Psi_\EM(E) = (x_E,\rho_E)$,
then $\TC_\EM(E) = \ic_\EM(x_E,\rho_E)$, and $E$ is the simple
$\EM W$-module corresponding to the local system
$j_\rs^*\FC^{-1}\ic_\EM(x_E,\rho_E)$. We have used the inverse
Fourier transform $\FC^{-1}$ of Theorem \ref{th:inv}. 
\end{proof}

\begin{definition}
The map $\Psi_\EM$ of Theorem \ref{th:springer} is the Springer
correspondence for $\EM W$ (by Fourier-Deligne transform).
\end{definition}

So the Springer correspondence $\Psi_\EM$ induces a bijection from
$\Irr \EM W$ onto its image. In the ordinary case, this image contains
all pairs involving a trivial local system. This is no longer true in
the modular case. However, the following proposition shows that the
trivial pair is always in the image.

\begin{proposition}
Let $\EM$ be $\KM$ or $\FM$. The Springer correspondence for $\EM W$
maps the trivial $\EM W$-module to the trivial nilpotent pair. That
is, if $1_W$ denotes the trivial $\EM W$-module, then we have
\[
\Psi_\EM(1_W) = (0,1).
\]
\end{proposition}

\begin{proof}
We have $\TC_\EM(1_W) = \FC(j_{\rs!*} \un \EM_{\gg_\rs}) = \FC(\un \EM_\gg) =
\ic_\EM(0,1)$ by Proposition \ref{prop:dirac} \eqref{eq:dirac}.
\end{proof}

Recall that the Springer correspondence by Fourier transform differs
from the Springer correspondence by restriction.  For example, we have
\[
\RC_\KM(1_W) = i_\NC^* j_{\rs!*} \un \KM_{\gg_\rs} [-r]
= i_\NC^* \KM_\gg [2\nu] = \KM_\NC [2\nu] = \ic_\KM(x_\reg,1)
\]
(see the remarks at the end of \ref{subsubsec:res}).

We note that $\TC_\EM$ induces an equivalence of categories from
$\EM W$-mod onto its essential image, since it is fully
faithful. However, the inclusion of this image into the category of
$G$-equivariant $\EM$-perverse sheaves on $\NC$ (as a full subcategory)
does not preserve Ext groups.

Let us remark that if $\ell$ does not divide the order of any of the finite
groups $A_G(x)$, $x\in\NC$, then all the group algebras $\FM A_G(x)$ are
semisimple, so for each $x$ there is a natural bijection
$\Irr \KM A_G(x) \stackrel{\sim}{\to} \Irr \FM A_G(x)$,
and thus there is a natural bijection
$\PG_\KM \stackrel{\sim}{\to} \PG_\FM$.

\section{Decomposition matrices and basic sets}
\label{sec:decomposition}

We have decomposition maps for representations of the Weyl group $W$ and for
perverse sheaves on the nilpotent cone $\NC$. In this section, we are going
to compare them. The main result is that the decomposition matrix of a Weyl
group is a submatrix of the decomposition matrix for $G$-equivariant
perverse sheaves on $\NC$.

Let us recall some results of Subsection \ref{subsec:IC} in the
context of the nilpotent cone. We have the decomposition matrix
$D^\NC$ for $G$-equivariant perverse sheaves on the nilpotent cone:
for $(x,\rho)\in\PG_\KM$ and $(y,\sigma)\in\PG_\FM$, recall that
\[
d^\NC_{(x,\rho),(y,\sigma)}
:= [\FM \ic(\ov\OC_x, \LC_{\rho,\OM}) : \ic(\ov\OC_y, \LC_\sigma)],
\]
where $\OC_x$ (resp. $\OC_y$) is the orbit of $x$ (resp. $y$), and
$\LC_\rho$ (resp. $\LC_\sigma$) is the local system associated to
$\rho$ (resp. $\sigma$). Recall that this number does not depend on
the choice of an integral form $\LC_{\rho,\OM}$.
Those decomposition numbers satisfy the following block triangularity
properties:
\begin{gather}\label{eq:unitri}
d^\NC_{(x,\rho),(x,\sigma)} = d^{A_G(x)}_{\rho,\sigma};\\
d^\NC_{(x,\rho),(y,\sigma)} \neq 0 \Imp \OC_y \leq \OC_x.
\end{gather}

Actually, the decomposition map $d$ is only one side of the
Brauer-Nesbitt $cde$-triangle. We will first see that the $e$
coefficients can be interpreted as multiplicities of simple perverse
sheaves in the perverse sheaves $\TC(\KM P_F)$ (we do not claim that
these perverse sheaves are projective, however).

After that, we will see that the Springer correspondence allows us to
define \emph{Springer basic sets} for Weyl groups. This gives a
geometric proof that the decomposition matrices of Weyl groups are
unitriangular. Geck and Rouquier have defined \emph{canonical basic
  sets} for Hecke algebras, under some assumptions on the
characteristic, using Lusztig's $a$ function \cite{GeckRouquier}. It would be interesting
to compare both approaches.

\subsection{Comparison of $e$ maps}

\begin{theorem}\label{th:e}
Let $F \in \Irr \FM W$. Then $\TC(\KM P_F)$ is supported on $\NC$, and for each
$E \in \Irr \KM W$  we have
\[
[\KM P_F : E] = [\TC(\KM P_F) : \TC(E)].
\]
\end{theorem}

\begin{proof}
We have
\[
\bigoplus_{F \in \FM W} \TC(\KM P_F)^{\dim F}
= \TC(\KM W) = \KM\KC_\NC 
\]
hence $\TC(\KM P_F)$ is supported on $\NC$. Moreover,
\[
\begin{array}{rcll}
[\KM P_F : E]
&=& [{j_\rs}_{!*} \KM \un P_F : {j_\rs}_{!*} \un E]
&\text{by Corollary \ref{cor:IC mult}}
\\
&=& [\FC {j_\rs}_{!*} \KM \un P_F (\nu + r) : \FC {j_\rs}_{!*} \un E (\nu + r)]
\\
&=& [\KM \TC(P_F) : \TC(E)]
\end{array}
\]
\end{proof}

\subsection{Comparison of $d$ maps}

If $E \in \Irr \KM W$ and $F \in \Irr \FM W$, let $d^W_{E,F}$ be the
corresponding decomposition number.

\begin{theorem}\label{th:d}
Let $E \in \Irr \KM W$, and let $E_\OM$ be an integral form for $E$.
Then $\TC(E_\OM)$ is supported on $\NC$, and for each $F \in \Irr \FM W$
we have
\[
[\FM E_\OM : F] = [\FM \TC(E_\OM) : \TC(F)]
\]
Thus
\[
d^W_{E,F} = d^\NC_{\Psi_\KM(E),\Psi_\FM(F)}
\]
\end{theorem}

\begin{proof}
By \cite[(2.57)]{decperv}, we have a short exact sequence
\begin{equation}\label{T}
0 \longto T \longto \FM {j_\rs}_{!*} \un E_\OM
\longto {j_\rs}_{!*} \FM \un E_\OM \longto 0
\end{equation}
with $T$ supported on $\gg - \gg_\rs$.

We have $\KM \TC(E_\OM) = \TC(E)$, so it is supported by $\NC$.
Now $\TC(E_\OM)$ is torsion-free because $E_\OM$ is torsion-free,
${j_\rs}_{!*}$ preserves monomorphisms, and $\FC$ is an equivalence.
Hence by Proposition \ref{prop:int supp F} it is also supported by $\NC$.
Moreover,
\[
\begin{array}{cll}
\multicolumn{2}{l}{[\FM E_\OM : F]}\\
=& [{j_\rs}_{!*} \FM \underline E_\OM : {j_\rs}_{!*} \underline F]
&\text{by Corollary \ref{cor:open mult}}
\\
=& [\FM {j_\rs}_{!*} \underline E_\OM : {j_\rs}_{!*} \underline F]
&\text{by (\ref{T}) and } [T : {j_\rs}_{!*} \underline F] = 0
\\
=& [\FM \FC({j_\rs}_{!*} \underline E_\OM) (\nu + r) : \FC({j_\rs}_{!*} \underline F)(\nu + r)]
&\text{by EQUIV and } \FC\FM = \FM\FC
\\
=& [\FM \TC(E_\OM) : \TC(F)].
\end{array}
\]
\end{proof}

This theorem means that we can obtain the decomposition matrix of the
Weyl group $W$ by extracting certain rows (the image $\impsi_\KM$ of
the ordinary Springer correspondence) and certain columns (the image
$\impsi_\FM$ of the modular Springer correspondence) of the
decomposition matrix for $G$-equivariant perverse sheaves on the
nilpotent cone $\NC$.

\subsection{Springer basic sets}
\label{subsec:basicsets}

Recall that an $\ell$-modular basic set of a finite group is a set of Brauer characters which is a basis of
the Grothendieck group of modular representations. Such a basic set is said to be ordinary if those
Brauer characters are restrictions of ordinary characters to the set of $\ell$-regular elements.
Then this set of ordinary characters is usually called a basic set itself.
Typically, one looks for ordinary basic sets such that the corresponding submatrix of the decomposition matrix
is unitriangular (for a certain order on ordinary characters). Such a ``triangular basic set''
allows one to define an injection of Brauer characters to ordinary characters.

For a nilpotent orbit $\OC$, let us denote by $A_G(\OC)$ the group
$A_G(x_\OC)$ for some $x_\OC \in \OC$. Since we assume that $G$ is
simple adjoint, each group $A_G(\OC)$ is either $\SG_2^k$, $\SG_3$,
$\SG_4$ or $\SG_5$. In the case of a symmetric group, the ordinary irreducible
representations are parametrized by partitions, they are naturally
ordered by the dominance order, and we have a well-known triangular basic set given
by the $\ell$-regular partitions. In the case of $\SG_2^k$, we take
the $k$-fold cartesian product of the ordered set $(2) > (1^2)$. In
the case $\ell = 2$, we take the obvious basic set consisting of the
trivial representation (corresponding to the choice of $(2)$ in each factor).
For each nilpotent orbit $\OC$, let
$\beta_\OC : \Irr \FM A_G(\OC) \injto \Irr \KM A_G(\OC)$ denote the injections
corresponding to the basic sets mentioned above. We denote by
$\beta_\NC : \PG_\FM \injto \PG_\KM$ the injection obtained by taking
the disjoint union of all $\beta_\OC$. Recall that, if $\ell$ does not
divide the orders of the $A_G(\OC)$, then this is actually a bijection.

In the following diagram, we would like to fill in the dashed line to
obtain a \emph{Springer basic set} given by the injection $\beta_S$.

\begin{theorem}

There exists an injection $\beta_S : \Irr \FM W \injto \Irr \KM W$
making the following diagram commutative:
\[
\xymatrix{
\Irr \FM W \ar@<-0.5ex>@{^{(}->}[rr]^{\Psi_\FM}
\ar@<-0.5ex>@{^{(}-->}[d]^{\beta_S} &&
\PG_\FM \ar@<-0.5ex>@{^{(}->}[d]^{\beta_\NC}
\\
\Irr \KM W \ar@<-0.5ex>@{^{(}->}[rr]^{\Psi_\KM} &&
\PG_\KM
}
\]
It has the following properties (unitriangularity):
\begin{gather}
\forall F \in \Irr \FM W,\ d^W_{\beta_S(F),F} = 1\\
\forall E \in \Irr \KM W,\ \forall F\in \Irr \FM W,\ 
d^W_{E,F} \neq 0 \Imp E \leq F
\end{gather}
\end{theorem}

The first part can be reformulated as: each pair in the image of the
modular Springer correspondence goes under $\beta_\NC$ to a pair which
is in the image of the ordinary Springer correspondence. The second
part says that we can arrange the rows and columns of the
decomposition matrix of $W$ so that it is unitriangular, using the
order on ordinary characters induced by the order on nilpotent pairs
described above, and the map $\beta_S$. In other words, we have
constructed a triangular basic set using the Springer correspondence. We call it
the \emph{Springer basic set}.
To achieve this aim, we first need the following result.

\begin{proposition}
\label{prop:imspri}
For $(x,\rho) \in \PG_\KM$ and $(y,\sigma)\in\PG_\FM$,
if $(x,\rho)\notin\im\Psi_\KM$ and
$d^\NC_{(x,\rho),(y,\sigma)} \neq 0$
then $(y,\sigma)\notin\im\Psi_\FM$.
\end{proposition}

\begin{proof}
Since $(x,\rho)\notin\im\Psi_\KM$, we have $j_\rs^*\FC^{-1}\ic_\KM(x,\rho) = 0$.
In other words, we are in the situation of Proposition \ref{prop:int supp F}
with $U = \gg_\rs$ and $Z = \gg \setminus \gg_\rs$, and
$\FC^{-1}\ic_\KM(x,\rho)$ is supported by $Z$. We deduce that its integral form
$\FC^{-1}\ic_\OM(x,\rho_\OM)$ (for any choice of $\rho_\OM$)
is also supported by $Z$, and also the modular reduction
$\FM\FC^{-1}\ic_\OM(x,\rho_\OM)$. Thus we have the following equality
in the Grothendieck group of $\FM$-local systems on $\gg_\rs$:
\[
\begin{array}{rcl}
[j_\rs^* \FM \FC^{-1} \ic_\OM(x,\rho_\OM)]
&=& [j_\rs^* \FC^{-1} \FM \ic_\OM(x,\rho_\OM)]\\
&=& \displaystyle\sum_{(y,\sigma)\in\PG_\FM} d^\NC_{(x,\rho),(y,\sigma)} [j_\rs^*\FC^{-1}\ic_\FM(y,\sigma)]
= 0.
\end{array}
\]
The $[j_\rs^*\FC^{-1}\ic_\FM(y,\sigma)]$ which are non-zero are the
classes of distinct simple objects, hence are linearly independent.
Now if $(y,\sigma) = \Psi_\FM(F)$, then 
$j_\rs^*\FC^{-1}\ic_\FM(y,\sigma) = \un F[2\nu+r] \neq 0$.
We can conclude that
$d^\NC_{(x,\rho),(y,\sigma)} = 0$ for all $(y,\sigma)\in\im\Psi_\FM$.
\end{proof}

\begin{proof}[Proof of theorem]
Let $F \in \Irr \FM W$. We want to prove that $\beta_\NC(\Psi_\FM(F))$
is in $\im \Psi_\KM$.  Let $(x,\sigma) := \Psi_\FM(F)$, and $\rho :=
\beta_x(\sigma)$ so that $(x,\rho) = \beta_\NC(x,\sigma)$. Then
$d^\NC_{(x,\rho),(x,\sigma)} = 1 \neq 0$, hence, by Proposition \ref{prop:imspri},
$(x,\rho) = \Psi_\KM(E)$ for some $E\in\Irr\KM W$ (which is uniquely
determined since $\Psi_\KM$ is injective). We define $\beta_S(F) := E$. 
The diagram commutes by construction.
Moreover, for $F \in \Irr \FM W$ with $\Psi_\FM(F) = (y,\sigma)$, we have
\[
d^W_{\beta_S(F),F} = d^\NC_{\Psi_\KM(\beta_S(F)),\Psi_\FM(F)}
= d^\NC_{\beta_\NC(\Psi_\FM(F)),\Psi_\FM(F)}
= d^\NC_{(y,\beta_y(\sigma)),(y,\sigma)}
= d^{A_G(y)}_{\beta_y(\sigma),\sigma}
= 1,
\]
and if in addition $E\in\Irr\EM W$ is such that
$d^W_{E,F} \neq 0$, then let $(x,\rho) = \Psi_\KM(E)$
\[
d^\NC_{(x,\rho),(y,\sigma)}
= d^\NC_{\Psi_\KM(E),\Psi_\FM(F)}
= d^W_{E,F}
\neq 0
\]
hence $x \leq y$, and in case of equality we have
$d^{A_G(x)}_{\rho,\sigma}\neq 0$, hence $\beta_x(\sigma)\leq \rho$.
By definition, this means $E\leq\beta_S(F)$.
\end{proof}

\section{Some decomposition numbers}
\label{sec:dec}

We refer to \cite{decperv} for some methods to compute decomposition numbers for
perverse sheaves. The calculations for the extreme cases (regular and
subregular orbits, or minimal and trivial orbits) were done there.
We assume that $G$ is almost simple, so that there is a unique
subregular nilpotent orbit and a unique minimal (non-zero) nilpotent orbit.

\subsection{Subregular nilpotent orbit}

The singularity of the nilpotent cone along the subregular orbit is a
simple singularity \cite{BRI}, meaning that the intersection of the nilpotent
cone with a transverse slice to the subregular orbit is isomorphic to 
the quotient of $\AM^2$ by a finite subgroup of $SL_2$. Those
singularities have an $ADE$ classification. We will denote by $x_\reg$
(resp. $x_\subreg$) a regular (resp. subregular) nilpotent element.

Let $\Gamma$ be the type of $\gg$.
If $\Gamma$ is simply laced, then the nilpotent cone has a simple
singularity of the same type $\Gamma$ along the subregular orbit.
The general case is explained in \cite{SLO2}.
One can associate to $\Gamma$ a simply-laced diagram $\wh
\Gamma$ and a subgroup $A$ of its automorphism group, so that $\Gamma$
can be seen as a kind of folding of $\wh\Gamma$. We assume that
$G$ is adjoint. Then this group $A$ can be identified with
$A_G(x_\subreg)$. Let $\wh\Phi$ be a root system of type
$\wh\Gamma$, with weight lattice $P(\wh\Phi)$ and root lattice
$Q(\wh\Phi)$. The group $A$ acts naturally on
$P(\wh\Phi)/Q(\wh\Phi)$. Then the singularity of the nilpotent cone along the
subregular orbit, along with the symmetries induced by
$A_G(x_\subreg)$, is a simple singularity of type $\wh\Gamma$, with
the action of $A$ by automorphisms.

Using this geometric description, the decomposition numbers
involving the regular and subregular orbits have been computed in
\cite[Section 4]{decperv}. The result is the following:

\begin{theorem}
We assume that $G$ is adjoint. Then the decomposition numbers for the
regular and subregular orbits are given by
\[
d^\NC_{(x_\reg,1),(x_\subreg,\rho)}
= \left[\FM \otimes_\ZM \left(P(\wh\Phi)/Q(\wh\Phi)\right) : \rho\right]
\]
where $\rho$ runs over the irreducible representations of $\FM A_G(x_\subreg)$.
\end{theorem}

\subsection{Minimal nilpotent orbit}

We denote the minimal nilpotent orbit of $\gg$ by $\OC_\mini$.
Its closure is $\ov\OC_\mini = \OC_\mini \cup \OC_\triv$, where
$\OC_\triv$ is the trivial orbit $\{0\}$. It is of dimension
$2h^\vee - 2$, where $h^\vee$ is the dual Coxeter number.

Remember that we denote by $\Phi$ the root system of $\gg$ with
respect to some Cartan subalgebra $\tg$. Now we choose some basis of
$\Phi$, and we denote by $\Phi'$ the root subsystem generated by the
long simple roots. Let $W'$ be the Weyl group of $\Phi'$.
We have an epimorphism $W \to W'$. Let $\nat_{W'}$ denote the
reflection representation of $W'$. We denote by $\chi_\mini$ its lift
to $W$. Then we have $\Psi_\KM(\chi_\mini) = (x_\mini,1)$.
Note that, if $\Phi$ is of simply-laced type, then $\Phi' = \Phi$,
$W' = W$, $\chi_\mini = \nat_W$.
On the other hand, we know that $\Psi_\FM(1_W) = (0,1)$, and of course the
trivial representation $1_W$ is the lift to $W$ of the trivial
representation $1_{W'}$ of $W'$.
By Theorem \ref{th:d} and \cite[Section 5]{decperv},
we have the following result:

\begin{theorem}
The decomposition numbers for the minimal and trivial orbits are given by:
\[
d^W_{\chi_\mini,1_W} =
d^{W'}_{\nat_{W'},1_{W'}} =
d^\NC_{(x_\mini,1),(0,1)} =
\dim_\FM \FM \otimes_\ZM \left( P^\vee(\Phi')/Q^\vee(\Phi') \right).
\]
\end{theorem}

\subsection{Special nilpotent orbits}
\label{subsec:special}

In \cite{LusSpec}, Lusztig introduced the special representations of a
finite Weyl group. The \emph{special nilpotent orbits} of $\gg$ 
are the nilpotent orbits $\OC$ such that the representation of
$\KM W$ corresponding (via restriction) to the pair $(\OC,\KM)$ is
\emph{special}. (The original definitions were in terms of unipotent
classes but of course one can work with nilpotent orbits instead.)

On the other hand, Spaltenstein introduced an order-reversing map
$d$ from the set of nilpotent orbits to itself, such that $d^3 = d$
(it is an involution on its image) in \cite{SpalDual}. The image of
$d$ consists exactly of the special nilpotent orbits, and the locally
closed subvarieties
\[
\widehat \OC = \ov \OC \setminus
\mathop{\bigcup_{\OC' \text{ special}}}\limits_{\ov \OC'\subset\ov \OC} \ov \OC'
\]
where $\OC$ runs through the special orbits, form a partition of the
nilpotent cone (any nilpotent orbit is contained in a
$\widehat \OC$ for a unique special orbit $\OC$).

In type $A$, all the nilpotent orbits are special, so the special
pieces are just the nilpotent orbits. Based on the particular case
of the minimal special orbits, Lusztig conjectured in \cite{LusGreen}
that special pieces are rationally smooth. This was confirmed in \cite{BeSpa} for exceptional types,
and in \cite{KP3} for classical types. More precisely, Kraft and Procesi proved the following
(they actually work with a complex group, but the result is certainly true for when the characteristic
$p$ of the base field is odd): 

\begin{theorem}
Let $\OC$ be a special nilpotent orbit of a classical
group. Define the special piece $\widehat \OC$ as above. Then $\widehat \OC$ consists of
$2^d$ conjugacy classes, where $d$ is the number of irreducible
components of $\widehat \OC \setminus \OC$. There is a smooth variety $Y$
with an action of the group $\SG_2^d$, and an isomorphism
\begin{equation}
Y / \SG_2^d \elem{\sim} \widehat \OC
\end{equation}
which identifies the stratification of $\widehat \OC$ with the
stratification of the quotient by isotropy groups. (These are the
$2^d$ subproducts of $\SG_2^d$). In particular $\widehat \OC$ is
rationally smooth.
\end{theorem}

Actually their result gives more information: by Proposition \ref{prop:E smooth},
the variety $\widehat \OC$ is not only $\KM$-smooth, but also
$\FM$-smooth for $\ell \neq 2$. The following corollary is valid when the base field is $\CM$, and
probably over a base field of characteristic $p\neq 2$.

\begin{corollary}
Let $\OC$ be a special nilpotent orbit in a
classical group, with corresponding special piece $\widehat\OC$.
If $\ell \neq 2$ then $\widehat \OC$ is $\FM$-smooth and,
for all orbits $\OC' \subset \widehat \OC \setminus \OC$, 
and all $\FM$-local systems $\LC$ on $\OC'$, we have
\[
d_{(\OC,\FM),(\OC',\LC)} = 0.
\]
\end{corollary}

Seeing the result of Kraft and Procesi, Lusztig conjectured that also in exceptional
types, any special piece $\widehat \OC$ is the quotient of a smooth variety by a precise
finite group $\Gamma_\OC$ (which is a subgroup of Lusztig's canonical quotient of
$A_G(\OC)$, see \cite{LusUnip}). In \cite{FJLS} we prove a closely
related result: if $\OC'$ denotes the minimal orbit in the special
piece $\widehat \OC$, and if $V$ denotes the reflection representation
of $\Gamma_\OC$ (which is always a Coxeter group), then the intersection of $\ov\OC$ with a
transversal slice to $\OC'$ is isomorphic to the quotient of $(V
\oplus V^*)^k$ by $\Gamma_\OC$, where $k$ is some integer (usually one
in exceptional types). Thus we can deduce a result like the above
corollary for exceptional types, for the primes $\ell$ not dividing
the order of $\Gamma_\OC$.

\section{The case of $GL_n$}
\label{sec:gl_n}

In this section, we study the case of $G = GL_n$.
Then $\gg = \gg\lg_n$ is the Lie algebra of all $n \times n$ matrices,
and $\NC$ is the closed subset of nilpotent matrices in the usual
sense, with $GL_n$ acting by conjugation.
Given a nilpotent matrix $x$ in $\NC$, the ordered list of the
sizes of Jordan blocks in the Jordan normal form of $x$ determines a
partition of $n$. This defines a bijection between the nilpotent
orbits and the set of partitions of $n$. We denote by $\OC_\l$ the
orbit corresponding to $\l$, and by $x_\l$ an element in $\OC_\l$.
We have $A_G(x_\l) = 1$ for all $\l \vdash n$, hence we will need to
consider only constant local systems. So both $\PG_\KM$ and $\PG_\FM$
are in natural bijection with the set $\PG_n$ of all partitions of
$n$. To simplify the notation, we set $\ic_\EM(\l) = \ic_\EM(x_\l,1)$ and
$d^\NC_{\l,\mu} := d^\NC_{(x_\l,1),(x_\mu,1)}$.

\subsection{Springer correspondence for $GL_n$}

Let us recall the parametrization of simple modules for the symmetric
group \cite{JAMESirr}.
The simple modules of $\KM \SG_n$ are the Specht modules $S^\lambda$,
where $\lambda$ runs over the partitions of $n$. Those have an
integral form $S^\lambda_\OM$ equipped with a symmetric bilinear form $\beta$
which is non-degenerate after extension of the scalars to
$\KM$. However the rank of those forms may drop after tensoring with
$\FM$. The module $D^\mu := \FM \otimes_\OM S^\mu_\OM / \Rad(\FM\otimes_\OM \beta)$
is non-zero exactly when $\mu$ is $\ell$-regular (that is, all parts
of $\mu$ have multiplicity strictly smaller than $\ell$). The modules
$D^\mu$, for $\mu$ running over all $\ell$-regular partitions of $n$,
form a complete set of representatives of the isomorphism classes of
$\FM\SG_n$-modules.

We note $d^{\SG_n}_{\lambda,\mu} := d^{\SG_n}_{S^\lambda,D^\mu}$.
The decomposition matrix of $\SG_n$ is
unitriangular with respect to the dominance order of partitions:
for $\lambda$ and $\mu$ two partitions of $n$, where $\mu$ is $\ell$-regular, we have
\begin{gather}
d^{\SG_n}_{\lambda,\lambda} = 1;\\
d^{\SG_n}_{\lambda,\mu} \neq 0 \Imp \lambda \leq \mu.
\end{gather}
Recall that we have a similar property for perverse
sheaves (see \eqref{eq:unitri}), again for the dominance order (now $\mu$ is not assumed to be $\ell$-regular)
\begin{gather}
d^{\NC}_{\lambda,\lambda} = 1;\\
d^{\NC}_{\lambda,\mu} \neq 0 \Imp \lambda \geq \mu.
\end{gather}
In characteristic zero, it is known that $\Psi_\KM$ is a bijection,
which sends the Specht module $S^\l$ to the pair $(x_{\l'},1)$,
where $\l'$ denotes the partition transposed to $\l$.
We will see that those facts are enough to determine
the modular Springer correspondence for $G = GL_n$.

\begin{theorem}\label{thm:gln}
Suppose $G = GL_n$. If $\mu$ is an $\ell$-regular partition, then we have
\[
\Psi_\FM(D^\mu) = (x_{\mu'},1)
\]
where $\mu'$ is the partition transposed to $\mu$. Thus we have,
for two partitions $\l$ and $\mu$, with $\mu$ $\ell$-regular,
$d^{\SG_n}_{\l,\mu} = d^\NC_{\l',\mu'}$.
\end{theorem}

\begin{proof}
Let $\mu$ be an $\ell$-regular partition. We want to show that
$\Psi_\FM(\mu) = \mu'$. On the one hand, we have
\[
d^\NC_{\mu',\Psi_\FM(\mu)} 
= d^\NC_{\Psi_\KM(\mu), \Psi_\FM(\mu)}
= d^{\SG_n}_{\mu,\mu} = 1 \neq 0
\]
which proves that $\Psi_\FM(\mu) \leqslant \mu'$.

On the other hand, we have
\[
d^{\SG_n}_{\Psi_\FM(\mu)',\mu}
= d^\NC_{\Psi_\KM(\Psi_\FM(\mu)'),\Psi_\FM(\mu)}
= d^\NC_{\Psi_\FM(\mu),\Psi_\FM(\mu)}
= 1 \neq 0
\]
which proves that $\Psi_\FM(\mu)' \leqslant \mu$, that is,
$\Psi_\FM(\mu) \geqslant \mu'$.
\end{proof}

Note that this result has been generalized in \cite{AHJRI}.

\subsection{Row and column removal rule}
\label{subsec:rc}

We introduce the ``characteristic functions''
\[
\begin{array}{rcl}
\chi_{\l,\mu}
&=& \sum_{i\in\ZM} (-1)^i \dim_\KM \HC^i_{x_\mu} \ic(\ov\OC_\l,\KM)\\
&=& \sum_{i\in\ZM} (-1)^i \dim_\FM \HC^i_{x_\mu} \FM \ic(\ov\OC_\l,\OM)
\end{array}
\]
and
\[
\phi_{\nu,\mu}
= \sum_{i\in\ZM} (-1)^i \dim_\FM \HC^i_{x_\mu} \ic(\ov\OC_\nu,\FM)
\]
These form triangular systems, in the sense that
$\chi_{\l,\mu}$ can be non-zero only if $\mu \leqslant \l$,
and $\phi_{\nu,\mu}$ can be non-zero only if $\mu \leqslant \nu$.
We have
\[
\chi_{\l,\mu} = \sum_{\nu} d_{\l,\nu} \phi_{\nu,\mu}
\]
and $d_{\l,\nu}$ can be non-zero only if $\nu \leqslant \l$.
Moreover, we have $\chi_{\l,\l} = \phi_{\l,\l} = 1$, and
$d_{\l,\l} = 1$.

Kraft and Procesi found a row and column removal rule for the
singularities of the closures of the nilpotent orbits in type
$A_{n-1}$ \cite{KP1}. Actually, they state the result when the base
field is $\CM$. It is probably true also in characteristic $p$, but
since this result is not available, for the rest of this section we
work over $\CM$. One can use the Fourier-Sato transform instead of the
Fourier-Deligne transform and everything goes through.

\begin{proposition}
Let $\OC_\mu \subset \ov\OC_\l$ be a degeneration of nilpotent orbits
in $\gg\lg_n$ and assume that the first $r$ rows and the first $s$
columns of $\l$ and $\mu$ coincide. Denote by $\lamh$ and $\muh$ the
Young diagrams obtained from $\l$ and $\mu$ by erasing these rows and
columns. Then $\OC_{\muh} \subset \ov\OC_{\lamh}$ in a smaller
nilpotent cone $\NCh \subset \gg\lg_{\hat n}$, and we have
\[
\codim_{\ov\OC_{\lamh}} \OC_{\muh} = \codim_{\ov\OC_\l} \OC_\mu
\quad\text{ and }\quad
\Sing(\ov\OC_{\lamh}, \OC_{\muh}) = \Sing(\ov\OC_\l, \OC_\mu)
\]
\end{proposition}

All the partitions $\nu$ in the interval $[\mu,\l] = \{\nu\mid \mu
\leqslant \nu \leqslant \l\}$ have the same first $r$ rows and first
$s$ columns. For $\nu$ in $[\mu,\l]$, let us denote by $\nuh$ the
partition obtained from $\nu$ by erasing them.

The proposition implies that, for all $\eta \leqslant \zeta$ in
$[\mu,\l]$, the local intersection cohomology of $\ov\OC_{\etah}$ along
$\OC_{\zeth}$ is the same as the local intersection cohomology of
$\ov\OC_\eta$ along $\OC_\zeta$, both over $\KM$ and over $\FM$, and
thus $\chi_{\eta,\zeta} = \chi_{\etah,\zeth}$ and
$\phi_{\etah,\zeth} = \phi_{\etah,\zeth}$.

Since the decomposition numbers can be deduced from this
information (for $GL_n$, we only have trivial $A_G(x)$), we find a row
and column removal rule for the decomposition numbers for
$GL_n$-equivariant perverse sheaves on the nilpotent cone of
$GL_n$.

\begin{proposition}\label{prop:rc}
With the notations above, we have
\begin{equation}
d^\NC_{\l,\mu} = d^\NCh_{\lamh,\muh}
\end{equation}
\end{proposition}

\begin{proof}
The decomposition numbers
$(d^\NC_{\eta,\zeta})_{\mu \leqslant \zeta \leqslant \eta \leqslant \l}$
are the unique solution of the triangular linear system
\[
\chi_{\eta,\zeta} = \sum_{\mu \leqslant \nu \leqslant \l} d^\NC_{\eta,\nu} \phi_{\nu,\zeta}
\]
whereas the decomposition numbers
$(d^\NCh_{\etah,\zeth})_{\muh \leqslant \zeth \leqslant \etah \leqslant \lamh}$
are the unique solution of the triangular linear system
\[
\chi_{\etah,\zeth} = \sum_{\mu \leqslant \nu \leqslant \l}
d^\NCh_{\etah,\nuh} \phi_{\nuh,\zeth}
\]
Since we have $\chi_{\eta,\zeta} = \chi_{\etah,\zeth}$ and
$\phi_{\eta,\zeta} = \phi_{\etah,\zeth}$, the two systems are
identical, so that $d^\NC_{\eta,\zeta} = d^\NCh_{\etah,\zeth}$ for all
$\zeta \leqslant \eta$ in $[\mu,\l]$. In particular, we have
$d^\NC_{\l,\mu} = d^\NCh_{\lamh,\muh}$.
\end{proof}

\section{The Grothendieck and Springer sheaves for $SL_2$}
\label{sec:sl2}

We will completely describe $\KC$ and $\KC_\NC$ in the case $G = SL_2$.
Apart from the open stratum $\gg_\rs$, we just
have the two nilpotent orbits $\OC_\reg = \OC_\mini = \OC_{(2)}$
and $\OC_\triv = \OC_\subreg = \OC_{(1^2)} = \{0\}$.
On $\gg_\rs$, we will only consider local systems which become trivial
after a pullback by $\pi_\rs$.
We have $W = \SG_2$. The local system ${\pi_\rs}_* \EM$ corresponds to
the regular representation $\EM\SG_2$.

In characteristic $0$, the group algebra is semi-simple, and the perverse sheaf
$\KM\KC_\rs$ splits as the sum of the constant perverse sheaf $C^\rs$ and
the shifted local system $C_\e^\rs$ corresponding to the sign
representation of $\SG_2$. These two simple components are sent by
${j_\rs}_{!*}$ on two simple perverse sheaves on $\gg$, the constant
perverse sheaf $C$ (since $\gg$ is smooth), and the other one, $C_\e$.
Let us denote by $A$ the simple perverse sheaf supported on $\{0\}$,
and by $B$ the simple perverse sheaf $\ic(\ov\OC_\reg,\KM)$.
Since $\FC(C) = A$, we must have $\FC(C_\e) = B$. This gives the
Springer correspondence for $\sg\lg_2$ by Fourier transform.

Let us make tables for the stalks of the perverse sheaves involved.
We have a line for each stratum, and one column for each cohomology
degree. If $x$ is a point of a given stratum $\OC$ and $i$ is an integer,
the corresponding entry in the table of a perverse sheaf $\AC$ will be
the class of $\HC^i_x \AC$, seen as a representation of a suitable
group $A(\OC)$, in the Grothendieck group of $\EM A(\OC)$.
There is a column $\chi$ describing the alternating sum of the stalks
of each stratum.

Let us first describe $\KM\KC$. So, over $\gg_\rs$, we have the regular
representation of $\SG_2$. Over $\OC_\reg$, the fibers are single
points, so the cohomology of $\BC_{x_\reg}$ is just $\KM$.
But we have $\BC_0 = \BC = G/B = \PM^1$. We get the following table for
$\KM\KC$.
\[
\begin{array}{|c|c|c|c|c|c|c|}
\hline
\text{Stratum} & \text{Dimension} & \chi & -3 & -2 & -1 & 0\\
\hline
\gg_\rs & 3 & -1 - \e & \KM \oplus \KM_\e&.&.&.\\
\hline
\OC_\reg & 2 & -1 & \KM &.&.&.\\
\hline
\OC_\triv & 0 & -2 & \KM & . & \KM &.
\\
\hline 
\end{array}
\]
It is the direct sum of the two simple perverse sheaves $C$
\[
\begin{array}{|c|c|c|c|c|c|c|}
\hline
\text{Stratum} & \text{Dimension} & \chi & -3 & -2 & -1 & 0\\
\hline
\gg_\rs & 3 & -1  & \KM &.&.&.\\
\hline
\OC_\reg & 2 & -1 & \KM &.&.&.\\
\hline
\OC_\triv & 0 & -1 & \KM & . & . & .
\\
\hline 
\end{array}
\]
and $C_\e$, which we deduce by subtraction (we have a direct sum!)
\[
\begin{array}{|c|c|c|c|c|c|c|}
\hline
\text{Stratum} & \text{Dimension} & \chi & -3 & -2 & -1 & 0\\
\hline
\gg_\rs & 3 & - \varepsilon & \KM_\varepsilon&.&.&.\\
\hline
\OC_\reg & 2 & 0 & . &.&.&.\\
\hline
\OC_\triv & 0 & -1 & . & . & \KM &.\\
\hline 
\end{array}
\]
The simple $G$-equivariant perverse sheaves on $\NC$ are
$B = \ic(\ov\OC_\reg,\KM)$,
\[
\begin{array}{|c|c|c|c|c|c|c|}
\hline
\text{Stratum} & \text{Dimension} & \chi & -3 & -2 & -1 & 0\\
\hline
\gg_\rs & 3 & 0 & .&.&.&.\\
\hline
\OC_\reg & 2 & 1 & . & \KM &.&.\\
\hline
\OC_\triv & 0 & 1 & . & \KM &.&.\\
\hline 
\end{array}
\]
$A = \ic(\ov\OC_\triv,\KM)$
\[
\begin{array}{|c|c|c|c|c|c|c|}
\hline
\text{Stratum} & \text{Dimension} & \chi & -3 & -2 & -1 & 0\\
\hline
\gg_\rs & 3 & 0 & .&.&.&.\\
\hline
\OC_\reg & 2 & 0 & . & . &.&.\\
\hline
\OC_\triv & 0 & 1 & . & . &.& \KM\\
\hline 
\end{array}
\]
and the cuspidal $B_\e = \ic(\ov\OC_\reg,\KM_\e)$,
which is clean (its intermediate extension is
just the extension by zero), and stable by the
Fourier-Deligne transform, by the general theory
\[
\begin{array}{|c|c|c|c|c|c|c|}
\hline
\text{Stratum} & \text{Dimension} & \chi & -3 & -2 & -1 & 0\\
\hline
\gg_\rs & 3 & 0 & .&.&.&.\\
\hline
\OC_\reg & 2 & 1 & . & \KM_\e &.&.\\
\hline
\OC_\triv & 0 & 1 & . & . &.&.\\
\hline 
\end{array}
\]

We can check that, applying $i_\NC^*[-1]$ to $\KC$, we recover $\KC_\NC$.
This functor sends $C$ to $B$ and $C_\e$ to $A$. There is a twist by
the sign character between the two versions of the Springer
representations (by Fourier-Deligne transform, and by restriction).

So, to summarize the situation over $\KM$, we have
\[
\begin{array}{|c|c|c|}
\hline
\gg_\rs & \gg & \NC\\
\hline
C^\rs \oplus C^\rs_\e & C \oplus C_\e & B \oplus A\\
\hline
\end{array}
\]

If $\ell \neq 2$, the situation over $\FM$ is similar. Now let us
assume that $\ell = 2$. Then the sign representation becomes
trivial. The regular representation is an extension of the trivial module by
itself, so similarly $\FM\KC_\rs$ is an extension of the constant $c_\rs$
(reduction of $C_\rs$) by itself.

Now $\FM\KC$ is as follows:
\[
\begin{array}{|c|c|c|c|c|c|c|}
\hline
\text{Stratum} & \text{Dimension} & \chi & -3 & -2 & -1 & 0\\
\hline
\gg_\rs & 3 & -2 & \FM \SG_2&.&.&.\\
\hline
\OC_\reg & 2 & -1 & \FM &.&.&.\\
\hline
\OC_\triv & 0 & -2 & \FM & . & \FM &.
\\
\hline 
\end{array}
\]

It must be made of the simple perverse sheaves $a$, $b$ and $c$, where
$c$ is the constant on $\gg$ (it is the reduction of $C$, and has the
same table with $\FM$ instead of $\KM$), $a$ is the constant on the
origin (the reduction of $A$), and $b = \ic(\ov\OC_\reg,\FM)$
has the following table
\[
\begin{array}{|c|c|c|c|c|c|c|}
\hline
\text{Stratum} & \text{Dimension} & \chi & -3 & -2 & -1 & 0\\
\hline
\gg_\rs & 3 & 0 & .&.&.&.\\
\hline
\OC^\reg & 2 & 1 & . & \FM &.&.\\
\hline
\OC^\triv & 0 & 0 & . & \FM &\FM&.\\
\hline 
\end{array}
\]

Looking at the $\chi$ functions, we see that
$[\FM\KC] = 2[c] + [b]$ in the Grothendieck group of
$\p\MC_G(\NC,\FM)$. We know that the top and the socle of $\FM\KC$
must be $c$, the intermediate extension of $c$, and that $b$ cannot
appear either in the top nor in the socle. Thus there is only one
possible Loewy structure:
\[
\FM\KC = 
\begin{array}{c}
c\\
b\\
c
\end{array}
\]

Similarly, we find
\[
\FM\KC_\NC = 
\begin{array}{c}
a\\
b\\
a
\end{array}
\]

Thus, as we already know, $\FC(c) = a$, but we also deduce that
$\FC(b) = b$.

The restriction functor $i^*_\NC$ sends $c$ to the reduction of $B$,
which has the following Loewy structure (by Section \ref{sec:dec} and the more detailed information
in \cite{decperv}): 
\[
\begin{array}{c}
b\\
a
\end{array}
\]

The reduction of $C_\e$ has structure
\[
\begin{array}{c}
c\\
b
\end{array}
\]
and it restricts to $a$ (the reduction of $A$, which is the restriction of $C_\e$).

So we have the following situation.

\[
\begin{array}{|c|c|c|}
\hline
\gg_\rs & \gg & \NC\\
\hline
\begin{array}{c}
c_\rs\\
c_\rs
\end{array}
&
\begin{array}{c}
c\\
b\\
c
\end{array}
&
\begin{array}{c}
a\\
b\\
a
\end{array}\\
\hline
\end{array}
\]

And we check that we can get $\KC_\NC$ either by Fourier-Deligne
transform, or by restriction. For the Springer correspondence, $b$ is
missing. In appears neither in the top nor in the socle. Thus it is said to be cuspidal.
For the modular generalized Springer correspondence in type $A$, we refer to \cite{AHJRI} (general linear group)
and \cite{AHJRII} (special linear groups).

Another way to involve $b$ would be to replace the regular representation of the symmetric group by
the direct sum of all induced modules from parabolic subgroups (whose endomorphism algebra is the Schur algebra).
What we lack here is the induction from $\SG_2$ to $\SG_2$, which gives the
trivial module, and hence the constant shifted local system $c_\rs$. If we start from the direct sum
of this constant local system with $\KC_\rs$, then taking the intermediate
extension and restricting to the nilpotent cone, we get:
\[
\begin{array}{|c|c|c|}
\hline
\gg_\rs & \gg & \NC\\
\hline
c_\rs
\oplus
\begin{array}{c}
c_\rs\\
c_\rs
\end{array}
&
c \oplus
\begin{array}{c}
c\\
b\\
c
\end{array}
&
\begin{array}{c}
b\\
a
\end{array}
\oplus
\begin{array}{c}
a\\
b\\
a
\end{array}\\
\hline
\end{array}
\]
where both $a$ and $b$ appear as the top constituants of the restriction.
The structure of the direct summands is exactly similar to that of
the projective modules of the Schur algebra of $GL_2$ in degree $2$. This is not surprising by the results of 
\cite{Mautner:schur}.

\section{Tables for exceptional groups}
\label{sec:exc}

In this Section we will determine explicitly the modular Springer correspondence for all exceptional groups.
Since we consider the non-generalized modular Springer correspondence, only local systems with trivial central
character appear, and we may assume that $G$ is simple and adjoint.

Theorem \ref{th:d} is a powerful tool to determine the modular Springer
correspondence, if one knows enough about the decomposition matrix of the Weyl group:
the case of $GL_n$ was treated this way in Theorem \ref{thm:gln}. 
In fact, the results in Subsection \ref{subsec:basicsets} imply that
we can determine explicitly the modular Springer correspondence for any
given type if the decomposition matrix of the Weyl group is known. Much better,
the knowledge of the ordinary character table of the Weyl group is actually enough:
consider the matrix of the restrictions of the characters to the $\ell$-regular
conjugacy classes, in some total order compatible with the Springer correspondence
(nilpotent orbits are ordered by closure inclusions, and we choose a suitable order on the characters
of each component group). Then the Springer basic set consists of the characters whose restriction is
not in the linear span of the preceding restrictions (i.e. the lines making the rank increase by 1, from top to bottom).
Let us give the example of $G_2$ for $\ell = 2$ to illustrate this: 

\[
\tiny
\begin{array}{cc|*{4}c}
0&\chi_{1,0}&1&1\\
A_1&\chi_{1,3}'&1&1\\
\tilde A_1&\chi_{2,2}&2&-1\\
G_2(a_1),3&\chi_{2,1}&2&-1\\
G_2(a_1),21&\chi_{1,3}''&1&1\\
G_2&\chi_{1,6}&1&1\\
\end{array}
\]

The notation for the characters of the Weyl group is the standard one: the first integer denotes the degree,
and the second one denotes the first symmetric power of the natural representation where the character appears.
The notation for pairs is as follows: if the component group is trivial, we only write the orbit;
otherwise it is a symmetric group, and we indicate the irreducible local system by a partition.
Of course the trivial character is in the image of the Springer correspondence. The next character
$\chi_{1,3}'$ has the same restriction as the trivial character, hence it is not in the Springer basic set.
So the orbit $A_1$ (with the trivial local system) is not in the image of the modular Springer correspondence.
However the character $\chi_{2,2}$ has a restriction which is not in the span of the preceding lines,
hence it is in the Springer basic set. The modular reduction is the sum of a new Brauer character and possibly
some copies of the trivial one (actually none, since they are not in the same block; but for the sake of explaining
this method, let us pretend that we do not know it). Despite this ambiguity, this
is already enough information to characterize this new Brauer character (although we do not know it explicitly),
and to know that it is sent to $\tilde A_1$ via the modular Springer correspondence.

We could have proceeded that way, however
for Weyl groups of exceptional types, the decomposition matrices are completely known anyway
(see \cite{Khos, KhosM} for $\ell > 2$, and they are obtainable by GAP3 in all exceptional types for all characteristics),
so we will display the information in the form of those decomposition matrices.
The image of the modular Springer correspondence will be given by the top $1$'s in each column.
As usual, it is more convenient to display the information block by block. Blocks of defect zero
behave just as in characteristic zero. For blocks of defect one, one can display the information in the form
of a Brauer tree (which here is a line, because all characters are rational, hence real; besides, there is no exceptional vertex). For each block of defect one, there is one pair not appearing in the modular Springer correspondence.
We felt that, even if the decomposition matrices were already known, it was useful to include them for several reasons:
we display them all together, blockwise, with the characters ordered via the Springer correspondence, with labels
both in terms of characters and in terms of orbits and local systems, with a uniform and standard notation. In this way, the ordinary and modular Springer correspondences can be read off easily from them.

\subsection{Type $G_2$}

\subsubsection{Case $\ell = 2$}
We have the principal block, of defect two:
\[
\tiny
\begin{array}{cc|*{3}c}
0&\chi_{1,0}&1\\
A_1&\chi_{1,3}'&1\\
G_2(a_1),21&\chi_{1,3}''&1\\
G_2&\chi_{1,6}&1\\
\end{array}
\]
and another block of defect one:
\[
\tiny
\begin{array}{cc|*{3}c}
\tilde A_1&\chi_{2,2}&1\\
G_2(a_1),3&\chi_{2,1}&1\\
\end{array}
\]
This confirms that the Springer basic set is $\{\chi_{1,0}, \chi_{2,2}\}$, and that the image
of the modular Springer correspondence is $\{0, \tilde A_1\}$.

\subsubsection{Case $\ell = 3$}
We have two blocks of defect one:
\[
\tiny
\begin{array}{cc|cc c cc|cc}
0&\chi_{1,0}&1&.  & \qquad\qquad & A_1&\chi_{1,3}'&1&.\\
\tilde A_1&\chi_{2,2}&1&1 & \text{\normalsize and} & G_2(a_1),3&\chi_{2,1}&1&1\\
G_2&\chi_{1,6}&.&1 & \qquad\qquad & G_2(a_1),21&\chi_{1,3}''&.&1\\
\end{array}
\]
This could be summarized by the Brauer trees:
\[
\tiny
\xymatrix@C=.4cm{\chi_{1,0} \ar@{-}[r]& \chi_{2,2} \ar@{-}[r]& \chi_{1,6}} \qquad \text{\normalsize and} \qquad
\xymatrix@C=.4cm{\chi_{1,3}' \ar@{-}[r]& \chi_{2,1} \ar@{-}[r]& \chi_{1,3}''}
\]
The corresponding pairs are:
\[
\tiny
\xymatrix@C=.4cm{0 \ar@{-}[r]& \tilde A_1 \ar@{-}[r]& G_2} \qquad \text{\normalsize and} \qquad
\xymatrix@C=.4cm{A_1 \ar@{-}[r]& (G_2(a_1),3) \ar@{-}[r]& (G_2(a_1),21)}
\]
The Springer basic set is $\{\chi_{1,0}, \chi_{2,2}, \chi_{1,3}', \chi_{2,1}\}$. The associated
irreducible Brauer characters correspond to the pairs $0$, $\tilde A_1$, $A_1$ and $(G_2(a_1), 3)$.
The missing pairs are at the end of the trees (on the side of the largest nilpotent orbit).
Note that the modular pair $(G_2(a_1),21)$ is the modular reduction of the cuspidal pair from characteristic zero
coefficients, hence we already knew it should not be involved in the modular Springer correspondence.

\subsection{Type $F_4$}

\subsubsection{Case $\ell = 2$}
We only have the principal block. The image of the modular Springer correspondence is $\{0, A_1, (\tilde A_1, 2),
A_1 + \tilde A_1\}$.
\[
\tiny
\begin{array}{cc|*{6}c}
0&\chi_{1,0}&1&.&.&.\\
A_1&\chi_{2,4}'&.&1&.&.\\
\tilde A_1,2&\chi_{4,1}&.&1&1&.\\
\tilde A_1,11&\chi_{2,4}''&.&.&1&.\\
A_1{+}\tilde A_1&\chi_{9,2}&1&1&1&1\\
A_2,2&\chi_{8,3}'&2&1&.&1\\
A_2,11&\chi_{1,12}'&1&.&.&.\\
\tilde A_2&\chi_{8,3}''&2&.&1&1\\
A_2{+}\tilde A_1&\chi_{4,7}'&.&1&1&.\\
B_2,2&\chi_{9,6}'&1&1&1&1\\
\tilde A_2{+}A_1&\chi_{6,6}''&2&.&.&1\\
B_2,11&\chi_{4,8}&.&.&.&1\\
C_3(a_1),2&\chi_{16,5}&.&2&2&2\\
C_3(a_1),11&\chi_{4,7}''&.&1&1&.\\
F_4(a_3),4&\chi_{12,4}&.&2&2&1\\
F_4(a_3),31&\chi_{9,6}''&1&1&1&1\\
F_4(a_3),22&\chi_{6,6}'&2&.&.&1\\
F_4(a_3),211&\chi_{1,12}''&1&.&.&.\\
C_3&\chi_{8,9}''&2&1&.&1\\
B_3&\chi_{8,9}'&2&.&1&1\\
F_4(a_2),2&\chi_{9,10}&1&1&1&1\\
F_4(a_2),11&\chi_{2,16}'&.&.&1&.\\
F_4(a_1),2&\chi_{4,13}&.&1&1&.\\
F_4(a_1),11&\chi_{2,16}''&.&1&.&.\\
F_4&\chi_{1,24}&1&.&.&.\\
\end{array}
\]

\subsubsection{Case $\ell = 3$}
We have two blocks of defect two:
\[
\tiny
\begin{array}{cc|*{6}c}
0&\chi_{1,0}&1&.&.&.\\
A_1&\chi_{2,4}'&1&1&.&.\\
\tilde A_1,11&\chi_{2,4}''&1&.&1&.\\
A_2,11&\chi_{1,12}'&.&1&.&.\\
B_2,11&\chi_{4,8}&1&1&1&1\\
F_4(a_3),211&\chi_{1,12}''&.&.&1&.\\
F_4(a_2),11&\chi_{2,16}'&.&1&.&1\\
F_4(a_1),11&\chi_{2,16}''&.&.&1&1\\
F_4&\chi_{1,24}&.&.&.&1\\
\end{array}
\]
and
\[
\tiny
\begin{array}{cc|*{6}c}
\tilde A_1,2&\chi_{4,1}&1&.&.&.\\
A_2,2&\chi_{8,3}'&1&1&.&.\\
\tilde A_2&\chi_{8,3}''&1&.&1&.\\
A_2{+}\tilde A_1&\chi_{4,7}'&.&1&.&.\\
C_3(a_1),2&\chi_{16,5}&1&1&1&1\\
C_3(a_1),11&\chi_{4,7}''&.&.&1&.\\
C_3&\chi_{8,9}''&.&.&1&1\\
B_3&\chi_{8,9}'&.&1&.&1\\
F_4(a_1),2&\chi_{4,13}&.&.&.&1\\
\end{array}
\]

There is one block of defect one, described by the following tree:
\[
\tiny
\xymatrix@C=.4cm{\chi_{6,6}'' \ar@{-}[r]&  \chi_{12,4} \ar@{-}[r]& \chi_{6,6}'}
\]
corresponding to the following pairs:
\[
\tiny
\xymatrix@C=.4cm{A_2{+}A_1 \ar@{-}[r]&  (F_4(a_3),4) \ar@{-}[r]& (F_4(a_3),22)}
\]

Finally, the other characters each form a block of defect zero:
\[
\tiny
\begin{array}{*{4}{c}}
\chi_{9,2}&\chi_{9,6}'&\chi_{9,6}''&\chi_{9,10}\\
A_1{+}\tilde A_1&(B_2,2)&(F_4(a_3),31)&(F_4(a_2),2)\\
\end{array}
\]

\subsection{Type $E_6$}

\subsubsection{Case $\ell = 2$}
We have the principal, block of defect $7$:
\[
\tiny
\begin{array}{cc|*{7}c}
0&\chi_{1,0}&1&.&.&.&.\\
A_1&\chi_{6,1}&.&1&.&.&.\\
2A_1&\chi_{20,2}&.&1&1&.&.\\
3A_1&\chi_{15,4}&1&1&.&1&.\\
A_2,2&\chi_{30,3}&2&1&1&1&.\\
A_2,11&\chi_{15,5}&1&.&1&.&.\\
2A_2&\chi_{24,6}&2&.&1&1&.\\
A_2{+}2A_1&\chi_{60,5}&.&1&1&.&1\\
A_3&\chi_{81,6}&1&3&1&1&1\\
2A_2{+}A_1&\chi_{10,9}&2&.&.&1&.\\
A_3{+}A_1&\chi_{60,8}&.&2&.&1&1\\
D_4(a_1),3&\chi_{80,7}&.&4&.&2&1\\
D_4(a_1),21&\chi_{90,8}&2&2&2&1&1\\
D_4(a_1),111&\chi_{20,10}&.&2&.&1&.\\
D_4&\chi_{24,12}&2&.&1&1&.\\
A_4&\chi_{81,10}&1&3&1&1&1\\
A_4{+}A_1&\chi_{60,11}&.&1&1&.&1\\
A_5&\chi_{15,16}&1&1&.&1&.\\
E_6(a_3),2&\chi_{30,15}&2&1&1&1&.\\
E_6(a_3),11&\chi_{15,17}&1&.&1&.&.\\
D_5&\chi_{20,20}&.&1&1&.&.\\
E_6(a_1)&\chi_{6,25}&.&1&.&.&.\\
E_6&\chi_{1,36}&1&.&.&.&.\\
\end{array}
\]
and a block of defect one:
\[
\tiny
\begin{array}{cc|*{3}c}
A_2{+}A_1&\chi_{64,4}&1\\
D_5(a_1)&\chi_{64,13}&1\\
\end{array}
\]

\subsubsection{Case $\ell = 3$}
We have the principal block, of defect four:
\[
\tiny
\begin{array}{cc|*{12}c}
0&\chi_{1,0}&1&.&.&.&.&.&.&.&.&.\\
A_1&\chi_{6,1}&1&1&.&.&.&.&.&.&.&.\\
2A_1&\chi_{20,2}&1&1&1&.&.&.&.&.&.&.\\
3A_1&\chi_{15,4}&.&.&1&1&.&.&.&.&.&.\\
A_2,2&\chi_{30,3}&.&1&.&.&1&.&.&.&.&.\\
A_2,11&\chi_{15,5}&.&1&.&.&.&1&.&.&.&.\\
A_2{+}A_1&\chi_{64,4}&.&1&1&.&1&1&1&.&.&.\\
2A_2&\chi_{24,6}&.&.&1&.&.&.&1&.&.&.\\
A_2{+}2A_1&\chi_{60,5}&.&1&1&1&1&.&1&1&.&.\\
2A_2{+}A_1&\chi_{10,9}&.&1&.&.&.&.&.&1&.&.\\
A_3{+}A_1&\chi_{60,8}&1&1&1&1&.&1&1&1&1&.\\
D_4(a_1),3&\chi_{80,7}&.&1&.&.&1&1&1&1&.&1\\
D_4(a_1),21&\chi_{90,8}&.&.&.&.&1&2&2&.&.&1\\
D_4(a_1),111&\chi_{20,10}&.&.&.&.&.&1&1&.&.&.\\
D_4&\chi_{24,12}&.&.&.&.&.&1&.&.&1&.\\
A_4{+}A_1&\chi_{60,11}&1&1&.&.&.&1&.&1&1&1\\
A_5&\chi_{15,16}&1&.&.&.&.&.&.&.&1&.\\
D_5(a_1)&\chi_{64,13}&.&.&.&.&.&1&1&1&1&1\\
E_6(a_3),2&\chi_{30,15}&.&.&.&.&.&.&.&1&.&1\\
E_6(a_3),11&\chi_{15,17}&.&.&.&.&.&.&1&1&.&.\\
D_5&\chi_{20,20}&.&.&.&1&.&.&.&1&1&.\\
E_6(a_1)&\chi_{6,25}&.&.&.&1&.&.&.&1&.&.\\
E_6&\chi_{1,36}&.&.&.&1&.&.&.&.&.&.\\
\end{array}
\]
and two blocks of defect zero:
\[
\tiny
\begin{array}{cc|c c cc|c}
A_3&\chi_{81,6}&1 & \qquad\text{\normalsize and}\qquad & A_4&\chi_{81,10}&1\\
\end{array}
\]

\subsubsection{Case $\ell = 5$}
We have two blocks of defect one:
\[
\tiny
\begin{array}{cc|cccc c cc|cccc}
0&\chi_{1,0}&1&.&.&. & \qquad\qquad & A_1&\chi_{6,1}&1&.&.&.\\
2A_2&\chi_{24,6}&1&1&.&. & \qquad\qquad & A_2{+}A_1&\chi_{64,4}&1&1&.&.\\
A_4&\chi_{81,10}&.&1&1&. & \text{\normalsize and} & A_3&\chi_{81,6}&.&1&1&.\\
D_5(a_1)&\chi_{64,13}&.&.&1&1 & \qquad\qquad & D_4&\chi_{24,12}&.&.&1&1\\
E_6(a_1)&\chi_{6,25}&.&.&.&1 & \qquad\qquad &  E_6&\chi_{1,36}&.&.&.&1\\
\end{array}
\]
The remaining characters have degree divisible by $5$, hence form
each a defect zero block on their own:
\[
 \tiny
\begin{array}{cc|cc|cc}
\chi_{20,2}&2A_1&\chi_{10,9}&2A_2{+}A_1&\chi_{60,11}&A_4{+}A_1\\
\chi_{15,4}&3A_1&\chi_{60,8}&A_3{+}A_1&\chi_{15,16}&A_5\\
\chi_{30,3}&(A_2,2)&\chi_{80,7}&(D_4(a_1),3)&\chi_{30,15}&(E_6(a_3),2)\\
\chi_{15,5}&(A_2,11)&\chi_{90,8}&(D_4(a_1),21)&\chi_{15,17}&(E_6(a_3),11)\\
\chi_{60,5}&A_2{+}2A_1&\chi_{20,10}&(D_4(a_1),111)&\chi_{20,20}&D_5\\
\end{array}
\]

\subsection{Type $E_7$}

\subsubsection{Case $\ell = 2$}
We have the principal block, of defect 10:
{\tiny
\[
\begin{array}{cc|*{9}c}
0&\chi_{1,0}&1&.&.&.&.&.&.\\
A_1&\chi_{7,1}&1&1&.&.&.&.&.\\
2A_1&\chi_{27,2}&1&2&1&.&.&.&.\\
3A_1''&\chi_{21,3}&1&1&1&.&.&.&.\\
3A_1'&\chi_{35,4}&1&2&1&1&.&.&.\\
A_2,2&\chi_{56,3}&2&3&2&1&.&.&.\\
A_2,11&\chi_{21,6}&1&1&1&.&.&.&.\\
4A_1&\chi_{15,7}&1&1&.&1&.&.&.\\
A_2{+}A_1,2&\chi_{120,4}&2&3&2&1&1&.&.\\
A_2{+}A_1,11&\chi_{105,5}&1&2&2&.&1&.&.\\
A_2{+}2A_1&\chi_{189,5}&3&4&3&1&1&1&.\\
A_3&\chi_{210,6}&4&5&4&1&1&1&.\\
A_2{+}3A_1&\chi_{105,6}&3&3&2&1&.&1&.\\
2A_2&\chi_{168,6}&2&2&3&.&1&1&.\\
(A_3{+}A_1)''&\chi_{189,7}&3&4&3&1&1&1&.\\
2A_2{+}A_1&\chi_{70,9}&2&1&1&.&.&1&.\\
(A_3{+}A_1)'&\chi_{280,8}&2&3&2&1&1&1&1\\
D_4(a_1),3&\chi_{315,7}&3&5&3&2&1&1&1\\
D_4(a_1),111&\chi_{35,13}&1&2&1&1&.&.&.\\
A_3{+}2A_1&\chi_{216,9}&2&3&2&1&.&1&1\\
D_4(a_1),21&\chi_{280,9}&6&6&5&1&1&2&.\\
D_4(a_1){+}A_1,11&\chi_{189,10}&3&4&3&1&1&1&.\\
D_4&\chi_{105,12}&3&3&2&1&.&1&.\\
D_4(a_1){+}A_1,2&\chi_{405,8}&5&7&5&2&1&2&1\\
A_3{+}A_2,2&\chi_{378,9}&4&5&4&2&1&2&1\\
A_3{+}A_2,11&\chi_{84,12}&2&2&1&1&.&1&.\\
A_4,11&\chi_{336,11}&4&6&4&2&1&1&1\\
A_3{+}A_2{+}A_1&\chi_{210,10}&2&3&1&2&.&1&1\\
A_4,2&\chi_{420,10}&6&8&5&3&1&2&1\\
D_4{+}A_1&\chi_{84,15}&2&2&1&1&.&1&.\\
A_5''&\chi_{105,15}&3&3&2&1&.&1&.\\
D_5(a_1),2&\chi_{420,13}&6&8&5&3&1&2&1\\
D_5(a_1),11&\chi_{336,14}&4&6&4&2&1&1&1\\
A_4{+}A_2&\chi_{210,13}&2&3&1&2&.&1&1\\
A_5{+}A_1&\chi_{70,18}&2&1&1&.&.&1&.\\
D_5(a_1){+}A_1&\chi_{378,14}&4&5&4&2&1&2&1\\
A_5'&\chi_{216,16}&2&3&2&1&.&1&1\\
D_6(a_2)&\chi_{280,17}&2&3&2&1&1&1&1\\
E_6(a_3),2&\chi_{405,15}&5&7&5&2&1&2&1\\
E_6(a_3),11&\chi_{189,17}&3&4&3&1&1&1&.\\
E_7(a_5),3&\chi_{315,16}&3&5&3&2&1&1&1\\
E_7(a_5),21&\chi_{280,18}&6&6&5&1&1&2&.\\
E_7(a_5),111&\chi_{35,22}&1&2&1&1&.&.&.\\
D_5&\chi_{189,20}&3&4&3&1&1&1&.\\
A_6&\chi_{105,21}&3&3&2&1&.&1&.\\
D_6(a_1)&\chi_{210,21}&4&5&4&1&1&1&.\\
D_5{+}A_1&\chi_{168,21}&2&2&3&.&1&1&.\\
E_7(a_4),2&\chi_{189,22}&3&4&3&1&1&1&.\\
E_7(a_4),11&\chi_{15,28}&1&1&.&1&.&.&.\\
E_6(a_1),2&\chi_{120,25}&2&3&2&1&1&.&.\\
D_6&\chi_{35,31}&1&2&1&1&.&.&.\\
E_6(a_1),11&\chi_{105,26}&1&2&2&.&1&.&.\\
E_6&\chi_{21,36}&1&1&1&.&.&.&.\\
E_7(a_3),2&\chi_{56,30}&2&3&2&1&.&.&.\\
E_7(a_3),11&\chi_{21,33}&1&1&1&.&.&.&.\\
E_7(a_2)&\chi_{27,37}&1&2&1&.&.&.&.\\
E_7(a_1)&\chi_{7,46}&1&1&.&.&.&.&.\\
E_7&\chi_{1,63}&1&.&.&.&.&.&.\\
\end{array}
\]
}
and a block of defect one:
\[
\tiny
\begin{array}{cc|*{3}c}
A_4{+}A_1,2&\chi_{512,12}&1\\
A_4{+}A_1,11&\chi_{512,11}&1\\
\end{array}
\]

\subsubsection{Case $\ell = 3$}
We have two blocks of defect four:
\[
\tiny
\begin{array}{cc|*{12}c}
0&\chi_{1,0}&1&.&.&.&.&.&.&.&.&.\\
3A_1'&\chi_{35,4}&1&1&.&.&.&.&.&.&.&.\\
A_2,11&\chi_{21,6}&.&.&1&.&.&.&.&.&.&.\\
A_2{+}A_1,2&\chi_{120,4}&1&.&1&1&.&.&.&.&.&.\\
A_3&\chi_{210,6}&.&.&1&1&1&.&.&.&.&.\\
A_2{+}3A_1&\chi_{105,6}&.&.&.&1&.&1&.&.&.&.\\
2A_2&\chi_{168,6}&1&1&.&1&.&.&1&.&.&.\\
(A_3{+}A_1)'&\chi_{280,8}&1&1&.&1&1&1&1&1&.&.\\
D_4&\chi_{105,12}&.&.&.&.&1&.&.&1&.&.\\
A_3{+}A_2,11&\chi_{84,12}&1&1&.&.&.&.&1&1&.&.\\
A_3{+}A_2{+}A_1&\chi_{210,10}&.&.&1&1&.&1&1&.&1&.\\
A_4,2&\chi_{420,10}&.&.&.&1&1&.&1&.&.&1\\
A_4{+}A_1,2&\chi_{512,12}&1&.&1&1&1&1&1&1&1&1\\
D_5(a_1),11&\chi_{336,14}&.&.&.&.&1&.&1&1&.&1\\
A_5{+}A_1&\chi_{70,18}&.&.&1&.&.&.&.&.&1&.\\
E_7(a_5),3&\chi_{315,16}&.&.&1&.&.&.&.&.&2&1\\
E_7(a_5),21&\chi_{280,18}&.&.&.&.&.&.&1&.&1&1\\
E_7(a_5),111&\chi_{35,22}&.&.&.&.&.&.&1&.&.&.\\
E_7(a_4),11&\chi_{15,28}&1&.&.&.&.&.&.&1&.&.\\
E_6(a_1),11&\chi_{105,26}&.&.&.&.&.&1&1&1&1&.\\
E_6&\chi_{21,36}&.&.&.&.&.&1&.&1&.&.\\
E_7(a_3),2&\chi_{56,30}&.&.&.&.&.&1&.&.&1&.\\
E_7(a_1)&\chi_{7,46}&.&.&.&.&.&1&.&.&.&.\\
\end{array}
\]
and
\[
\tiny
\begin{array}{cc|*{12}c}
A_1&\chi_{7,1}&1&.&.&.&.&.&.&.&.&.\\
3A_1''&\chi_{21,3}&1&1&.&.&.&.&.&.&.&.\\
A_2,2&\chi_{56,3}&1&.&1&.&.&.&.&.&.&.\\
4A_1&\chi_{15,7}&.&1&.&1&.&.&.&.&.&.\\
A_2{+}A_1,11&\chi_{105,5}&1&1&1&.&1&.&.&.&.&.\\
2A_2{+}A_1&\chi_{70,9}&.&.&1&.&.&1&.&.&.&.\\
D_4(a_1),3&\chi_{315,7}&.&.&2&.&.&1&1&.&.&.\\
D_4(a_1),111&\chi_{35,13}&.&.&.&.&1&.&.&.&.&.\\
D_4(a_1),21&\chi_{280,9}&.&.&1&.&1&.&1&.&.&.\\
A_4,11&\chi_{336,11}&.&1&.&.&1&.&1&1&.&.\\
D_4{+}A_1&\chi_{84,15}&.&1&.&1&1&.&.&.&1&.\\
A_5''&\chi_{105,15}&.&1&.&.&.&.&.&1&.&.\\
A_4{+}A_1,11&\chi_{512,11}&1&1&1&1&1&1&1&1&.&1\\
D_5(a_1),2&\chi_{420,13}&.&.&.&.&1&.&1&1&.&1\\
A_4{+}A_2&\chi_{210,13}&1&.&1&.&1&1&.&.&.&1\\
D_6(a_2)&\chi_{280,17}&1&1&.&1&1&.&.&1&1&1\\
A_6&\chi_{105,21}&1&.&.&.&.&.&.&.&.&1\\
D_6(a_1)&\chi_{210,21}&.&.&.&.&.&1&.&1&.&1\\
D_5{+}A_1&\chi_{168,21}&.&.&.&1&1&.&.&.&1&1\\
E_6(a_1),2&\chi_{120,25}&.&.&.&1&.&1&.&.&.&1\\
D_6&\chi_{35,31}&.&.&.&1&.&.&.&.&1&.\\
E_7(a_3),11&\chi_{21,33}&.&.&.&.&.&1&.&.&.&.\\
E_7&\chi_{1,63}&.&.&.&1&.&.&.&.&.&.\\
\end{array}
\]
four blocks of defect one:
\[
\tiny
\begin{array}{cc|*{4}c}
2A_1&\chi_{27,2}&1&.\\
A_5'&\chi_{216,16}&1&1\\
D_5&\chi_{189,20}&.&1\\
\end{array}
\]
\[
\tiny
\begin{array}{cc|*{4}c}
(A_3{+}A_1)''&\chi_{189,7}&1&.\\
A_3{+}2A_1&\chi_{216,9}&1&1\\
E_7(a_2)&\chi_{27,37}&.&1\\
\end{array}
\]
\[
\tiny
\begin{array}{cc|*{4}c}
D_4(a_1){+}A_1,11&\chi_{189,10}&1&.\\
D_5(a_1){+}A_1&\chi_{378,14}&1&1\\
E_7(a_4),2&\chi_{189,22}&.&1\\
\end{array}
\]
\[
\tiny
\begin{array}{cc|*{4}c}
A_2{+}2A_1&\chi_{189,5}&1&.\\
A_3{+}A_2,2&\chi_{378,9}&1&1\\
E_6(a_3),11&\chi_{189,17}&.&1\\
\end{array}
\]
and two blocks of defect zero:
\[
\tiny
\begin{array}{cc|c c cc|c}
D_4(a_1){+}A_1,2&\chi_{405,8}&1  & \qquad\text{\normalsize and}\qquad & E_6(a_3),2&\chi_{405,15}&1\\
\end{array}
\]

\subsubsection{Case $\ell = 5$}

There are 6 blocks of defect one, with following Brauer trees:
\[
\tiny
\xymatrix@R=0cm@C=.4cm{
\chi_{1,0}\ar@{-}[r]&\chi_{84,12}\ar@{-}[r]&\chi_{216,16}\ar@{-}[r]&\chi_{189,22}\ar@{-}[r]&\chi_{56,30}\\
\chi_{56,3}\ar@{-}[r]&\chi_{189,5}\ar@{-}[r]&\chi_{216,9}\ar@{-}[r]&\chi_{84,15}\ar@{-}[r]&\chi_{1,63}\\
\chi_{27,2}\ar@{-}[r]&\chi_{168,6}\ar@{-}[r]&\chi_{512,12}\ar@{-}[r]&\chi_{378,14}\ar@{-}[r]&\chi_{7,46}\\
\chi_{7,1}\ar@{-}[r]&\chi_{378,9}\ar@{-}[r]&\chi_{512,11}\ar@{-}[r]&\chi_{168,21}\ar@{-}[r]&\chi_{27,37}\\
\chi_{21,6}\ar@{-}[r]&\chi_{189,10}\ar@{-}[r]&\chi_{336,14}\ar@{-}[r]&\chi_{189,20}\ar@{-}[r]&\chi_{21,36}\\
\chi_{21,3}\ar@{-}[r]&\chi_{189,7}\ar@{-}[r]&\chi_{336,11}\ar@{-}[r]&\chi_{189,17}\ar@{-}[r]&\chi_{21,33}\\
}
\]
and the corresponding pairs are:
\[
\tiny
\xymatrix@R=0cm@C=.2cm{
0\ar@{-}[r]&(A_3{+}A_2,11)\ar@{-}[r]&A_5'\ar@{-}[r]&(E_7(a_4),2)\ar@{-}[r]&(E_7(a_3),2)\\
(A_2,2)\ar@{-}[r]&A_2{+}2A_1\ar@{-}[r]&A_3{+}2A_1\ar@{-}[r]&D_4{+}A_1\ar@{-}[r]&E_7\\
2A_1\ar@{-}[r]&2A_2\ar@{-}[r]&(A_4{+}A_1,2)\ar@{-}[r]&D_5(a_1){+}A_1\ar@{-}[r]&E_7(a_1)\\
A_1\ar@{-}[r]&(A_3{+}A_2,2)\ar@{-}[r]&(A_4{+}A_1,11)\ar@{-}[r]&D_5{+}A_1\ar@{-}[r]&E_7(a_2)\\
(A_2,11)\ar@{-}[r]&(D_4(a_1){+}A_1,11)\ar@{-}[r]&(D_5(a_1),11)\ar@{-}[r]&D_5\ar@{-}[r]&E_6\\
3A_1''\ar@{-}[r]&(A_3{+}A_1)''\ar@{-}[r]&(A_4,11)\ar@{-}[r]&(E_6(a_3),11)\ar@{-}[r]&(E_7(a_3),11)\\
}
\]

Finally, there are 30 blocks of defect zero:
\[
 \tiny
\begin{array}{cc|cc|cc}
\chi_{35,4}&3A_1'&\chi_{280,9}&(D_4(a_1),21)&\chi_{405,15}&(E_6(a_3),2)\\
\chi_{15,7}&4A_1&\chi_{105,12}&D_4&\chi_{315,16}&(E_7(a_5),3)\\
\chi_{120,4}&(A_2{+}A_1,2)&\chi_{405,8}&(D_4(a_1){+}A_1,2)&\chi_{280,18}&(E_7(a_5),21)\\
\chi_{105,5}&(A_2{+}A_1,11)&\chi_{210,10}&A_3{+}A_2{+}A_1&\chi_{35,22}&(E_7(a_5),111)\\
\chi_{210,6}&A_3&\chi_{420,10}&(A_4,2)&\chi_{105,21}&A_6\\
\chi_{105,6}&A_2{+}3A_1&\chi_{105,15}&A_5''&\chi_{210,21}&D_6(a_1)\\
\chi_{70,9}&2A_2{+}A_1&\chi_{420,13}&(D_5(a_1),2)&\chi_{15,28}&(E_7(a_4),11)\\
\chi_{280,8}&(A_3{+}A_1)'&\chi_{210,13}&A_4{+}A_2&\chi_{120,25}&(E_6(a_1),2)\\
\chi_{315,7}&(D_4(a_1),3)&\chi_{70,18}&A_5{+}A_1&\chi_{35,31}&D_6\\
\chi_{35,13}&(D_4(a_1),111)&\chi_{280,17}&D_6(a_2)&\chi_{105,26}&(E_6(a_1),11)\\
\end{array}
\]

So the following pairs are missing:
\[
E_7,\ E_7(a_1),\ E_7(a_2),\ (E_7(a_3),11),\ (E_7(a_3),2) \text{ and } E_6.
\]

\subsubsection{Case $\ell = 7$}

There are two blocks of defect one:
\[
\tiny
\xymatrix@R=.1cm@C=.4cm{
\chi_{1,0}\ar@{-}[r]&\chi_{27,2}\ar@{-}[r]&\chi_{120,4}\ar@{-}[r]&\chi_{405,8}\ar@{-}[r]&\chi_{512,12}\ar@{-}[r]&\chi_{216,16}\ar@{-}[r]&\chi_{15,28}\\
\chi_{15,7}\ar@{-}[r]&\chi_{216,9}\ar@{-}[r]&\chi_{512,11}\ar@{-}[r]&\chi_{405,15}\ar@{-}[r]&\chi_{120,25}\ar@{-}[r]&\chi_{27,37}\ar@{-}[r]&\chi_{1,63}\\
}
\]
and the corresponding pairs are:
\[
\tiny
\xymatrix@R=.1cm@C=.2cm{
0\ar@{-}[r]&2A_1\ar@{-}[r]&(A_2{+}A_1,2)\ar@{-}[r]&(D_4(a_1){+}A_1,2)\ar@{-}[r]&(A_4{+}A_1,2)\ar@{-}[r]&A_5'\ar@{-}[r]&(E_7(a_4),11)\\
4A_1\ar@{-}[r]&A_3{+}2A_1\ar@{-}[r]&(A_4{+}A_1,11)\ar@{-}[r]&(E_6(a_3),2)\ar@{-}[r]&(E_6(a_1),2)\ar@{-}[r]&E_7(a_2)\ar@{-}[r]&E_7\\
}
\]

There are 46 blocks of defect zero:
\[
 \tiny
\begin{array}{cc|cc|cc}
\chi_{7,1}&A_1&\chi_{189,10}&(D_4(a_1){+}A_1,11)&\chi_{315,16}&(E_7(a_5),3)\\
\chi_{21,3}&3A_1''&\chi_{105,12}&D_4&\chi_{280,18}&(E_7(a_5),21)\\
\chi_{35,4}&3A_1'&\chi_{378,9}&(A_3{+}A_2,2)&\chi_{35,22}&(E_7(a_5),111)\\
\chi_{56,3}&(A_2,2)&\chi_{84,12}&(A_3{+}A_2,11)&\chi_{189,20}&D_5\\
\chi_{21,6}&(A_2,11)&\chi_{336,11}&(A_4,11)&\chi_{105,21}&A_6\\
\chi_{105,5}&(A_2{+}A_1,11)&\chi_{210,10}&A_3{+}A_2{+}A_1&\chi_{210,21}&D_6(a_1)\\
\chi_{189,5}&A_2{+}2A_1&\chi_{420,10}&(A_4,2)&\chi_{168,21}&D_5{+}A_1\\
\chi_{210,6}&A_3&\chi_{84,15}&D_4{+}A_1&\chi_{189,22}&(E_7(a_4),2)\\
\chi_{105,6}&A_2{+}3A_1&\chi_{105,15}&A_5''&\chi_{35,31}&D_6\\
\chi_{168,6}&2A_2&\chi_{420,13}&(D_5(a_1),2)&\chi_{105,26}&(E_6(a_1),11)\\
\chi_{189,7}&(A_3{+}A_1)''&\chi_{336,14}&(D_5(a_1),11)&\chi_{21,36}&E_6\\
\chi_{70,9}&2A_2{+}A_1&\chi_{210,13}&A_4{+}A_2&\chi_{56,30}&(E_7(a_3),2)\\
\chi_{280,8}&(A_3{+}A_1)'&\chi_{70,18}&A_5{+}A_1&\chi_{21,33}&(E_7(a_3),11)\\
\chi_{315,7}&(D_4(a_1),3)&\chi_{378,14}&D_5(a_1){+}A_1&\chi_{7,46}&E_7(a_1)\\
\chi_{35,13}&(D_4(a_1),111)&\chi_{280,17}&D_6(a_2)&&\\
\chi_{280,9}&(D_4(a_1),21)&\chi_{189,17}&(E_6(a_3),11)&&\\
\end{array}
\]

The pairs $E_7$ and $(E_7(a_4),11)$ are missing.

\newpage

\subsection{Type $E_8$}

\subsubsection{Case $\ell = 2$}

The principal block has defect 14. To save space, we put on the same
line a character and its twist by the sign character. The left column shows the character of the pair which
appears first in the Springer order.
\[
\tiny
\begin{array}{cc|cc|*{11}c}
0&\chi_{1,0}&E_8&\chi_{1,120}&1&.&.&.&.&.&.&.&.&.&.\\
A_1&\chi_{8,1}&E_8(a_1)&\chi_{8,91}&.&1&.&.&.&.&.&.&.&.&.\\
2A_1&\chi_{35,2}&E_8(a_2)&\chi_{35,74}&1&1&1&.&.&.&.&.&.&.&.\\
3A_1&\chi_{84,4}&E_7&\chi_{84,64}&2&1&1&1&.&.&.&.&.&.&.\\
A_2,2&\chi_{112,3}&E_8(a_3),2&\chi_{112,63}&4&1&2&1&.&.&.&.&.&.&.\\
A_2,11&\chi_{28,8}&E_8(a_3),11&\chi_{28,68}&2&.&1&.&.&.&.&.&.&.&.\\
4A_1&\chi_{50,8}&E_8(b_4),11&\chi_{50,56}&.&1&1&.&1&.&.&.&.&.&.\\
A_2{+}A_1,2&\chi_{210,4}&E_8(a_4),2&\chi_{210,52}&.&1&1&.&1&1&.&.&.&.&.\\
A_2{+}A_1,11&\chi_{160,7}&E_8(a_4),11&\chi_{160,55}&.&.&.&.&.&1&.&.&.&.&.\\
A_2{+}2A_1&\chi_{560,5}&E_8(b_4),2&\chi_{560,47}&2&2&4&.&2&1&1&.&.&.&.\\
A_3&\chi_{567,6}&E_7(a_1)&\chi_{567,46}&3&2&3&1&1&1&1&.&.&.&.\\
A_2{+}3A_1&\chi_{400,7}&D_7&\chi_{400,43}&4&1&2&1&2&1&.&1&.&.&.\\
2A_2,2&\chi_{700,6}&E_8(a_5),2&\chi_{700,42}&6&1&4&1&2&1&1&1&.&.&.\\
2A_2,11&\chi_{300,8}&E_8(a_5),11&\chi_{300,44}&2&.&2&.&.&.&1&.&.&.&.\\
2A_2{+}A_1&\chi_{448,9}&E_6{+}A_1&\chi_{448,39}&6&.&2&1&.&.&1&1&.&.&.\\
A_3{+}A_1&\chi_{1344,8}&E_7(a_2)&\chi_{1344,38}&2&2&4&.&2&1&1&.&1&.&.\\
D_4(a_1),3&\chi_{1400,7}&E_8(b_5),3&\chi_{1400,37}&2&3&4&1&2&1&1&.&1&.&.\\
D_4(a_1),21&\chi_{1008,9}&E_8(b_5),21&\chi_{1008,39}&8&2&6&1&2&1&2&1&.&.&.\\
D_4(a_1),111&\chi_{56,19}&E_8(b_5),111&\chi_{56,49}&.&1&.&1&.&.&.&.&.&.&.\\
2A_2{+}2A_1&\chi_{175,12}&E_8(b_6),21&\chi_{175,36}&5&.&1&1&.&.&.&1&.&.&.\\
D_4&\chi_{525,12}&E_6&\chi_{525,36}&7&1&4&1&1&.&1&1&.&.&.\\
A_3{+}2A_1&\chi_{1050,10}&D_7(a_1),11&\chi_{1050,34}&4&3&4&1&3&1&1&1&.&1&.\\
D_4(a_1){+}A_1,3&\chi_{1400,8}&E_8(a_6),3&\chi_{1400,32}&6&4&7&1&4&1&2&1&.&1&.\\
D_4(a_1){+}A_1,21&\chi_{1575,10}&E_8(a_6),21&\chi_{1575,34}&7&3&5&2&2&1&1&1&1&.&.\\
D_4(a_1){+}A_1,111&\chi_{350,14}&E_8(a_6),111&\chi_{350,38}&2&1&3&.&1&.&1&.&.&.&.\\
A_3{+}A_2,2&\chi_{3240,9}&D_7(a_1),2&\chi_{3240,31}&10&7&10&3&6&2&3&2&1&2&.\\
A_3{+}A_2,11&\chi_{972,12}&D_6&\chi_{972,32}&2&2&2&1&2&1&1&1&.&1&.\\
A_4,2&\chi_{2268,10}&E_7(a_3),2&\chi_{2268,30}&8&5&8&2&4&1&2&1&1&1&.\\
A_4,11&\chi_{1296,13}&E_7(a_3),11&\chi_{1296,33}&6&3&6&1&2&.&1&.&1&.&.\\
A_3{+}A_2{+}A_1&\chi_{1400,11}&A_7&\chi_{1400,29}&2&5&4&1&4&1&1&1&.&2&.\\
D_4(a_1){+}A_2,2&\chi_{2240,10}&E_8(b_6),3&\chi_{2240,28}&2&6&4&2&4&1&1&1&1&2&.\\
D_4(a_1){+}A_2,11&\chi_{840,13}&E_8(b_6),111&\chi_{840,31}&.&1&.&1&.&.&.&.&1&.&.\\
D_4{+}A_1&\chi_{700,16}&E_7(a_4),11&\chi_{700,28}&2&2&1&1&2&1&.&1&.&1&.\\
2A_3&\chi_{840,14}&D_5{+}A_2,11&\chi_{840,26}&4&2&3&1&2&.&1&1&.&1&.\\
D_5(a_1),2&\chi_{2800,13}&E_6(a_1),2&\chi_{2800,25}&6&7&6&3&6&3&1&2&1&2&.\\
D_5(a_1),11&\chi_{2100,16}&E_6(a_1),11&\chi_{2100,28}&4&5&5&2&4&2&1&1&1&1&.\\
A_4{+}2A_1,2&\chi_{4200,12}&D_7(a_2),2&\chi_{4200,24}&14&6&11&3&6&1&2&3&1&1&1\\
A_4{+}2A_1,11&\chi_{3360,13}&D_7(a_2),11&\chi_{3360,25}&10&4&8&2&4&1&1&2&1&.&1\\
A_4{+}A_2&\chi_{4536,13}&D_5{+}A_2,2&\chi_{4536,23}&12&7&12&2&8&1&2&3&1&2&1\\
A_5&\chi_{3200,16}&D_5{+}A_1&\chi_{3200,22}&10&4&8&2&4&.&1&2&1&.&1\\
A_4{+}A_2{+}A_1&\chi_{2835,14}&A_6{+}A_1&\chi_{2835,22}&7&4&7&1&5&1&1&2&.&1&1\\
D_5(a_1){+}A_1&\chi_{6075,14}&E_7(a_4),2&\chi_{6075,22}&17&11&17&4&11&3&4&4&1&3&1\\
D_4{+}A_2,2&\chi_{4200,15}&A_6&\chi_{4200,21}&18&7&14&3&8&1&3&4&.&2&1\\
D_4{+}A_2,11&\chi_{168,24}&\multicolumn{2}{c|}{\text{sign-invariant}}&4&.&2&.&1&.&.&1&.&.&.\\
E_6(a_3),2&\chi_{5600,15}&D_6(a_1),2&\chi_{5600,21}&20&10&18&4&10&2&4&4&1&2&1\\
E_6(a_3),11&\chi_{2400,17}&D_6(a_1),11&\chi_{2400,23}&10&6&10&2&6&2&3&2&.&2&.\\
D_5&\chi_{2100,20}&\multicolumn{2}{c|}{\text{sign-invariant}}&8&6&8&2&6&2&2&2&.&2&.\\
A_4{+}A_3&\chi_{420,20}&\multicolumn{2}{c|}{\text{sign-invariant}}&.&2&2&.&2&.&.&.&.&1&.\\
A_5{+}A_1&\chi_{2016,19}&\multicolumn{2}{c|}{\text{sign-invariant}}&8&2&4&2&2&.&.&2&.&.&1\\
D_5(a_1){+}A_2&\chi_{1344,19}&\multicolumn{2}{c|}{\text{sign-invariant}}&.&4&.&2&4&2&.&2&.&2&.\\
E_6(a_3){+}A_1,2&\chi_{3150,18}&\multicolumn{2}{c|}{\text{sign-invariant}}&14&4&10&2&5&.&2&3&.&1&1\\
E_6(a_3){+}A_1,11&\chi_{1134,20}&\multicolumn{2}{c|}{\text{sign-invariant}}&6&2&6&.&3&.&2&1&.&1&.\\
D_6(a_2),2&\chi_{4200,18}&\multicolumn{2}{c|}{\text{sign-invariant}}&8&8&10&2&9&2&2&3&.&3&1\\
D_6(a_2),11&\chi_{2688,20}&\multicolumn{2}{c|}{\text{sign-invariant}}&4&4&8&.&4&.&2&.&.&1&1\\
E_7(a_5),3&\chi_{7168,17}&\multicolumn{2}{c|}{\text{sign-invariant}}&8&12&16&2&14&4&4&3&.&4&2\\
E_7(a_5),21&\chi_{5600,19}&\multicolumn{2}{c|}{\text{sign-invariant}}&12&10&12&4&8&2&2&3&2&2&1\\
E_7(a_5),111&\chi_{448,25}&\multicolumn{2}{c|}{\text{sign-invariant}}&.&.&.&.&2&2&.&1&.&.&.\\
E_8(a_7),5&\chi_{4480,16}&\multicolumn{2}{c|}{\text{sign-invariant}}&4&8&8&2&10&4&2&3&.&3&1\\
E_8(a_7),41&\chi_{5670,18}&\multicolumn{2}{c|}{\text{sign-invariant}}&14&10&14&4&9&2&2&3&2&2&1\\
E_8(a_7),311&\chi_{1680,22}&\multicolumn{2}{c|}{\text{sign-invariant}}&8&4&6&2&4&2&2&2&.&1&.\\
E_8(a_7),32&\chi_{4536,18}&\multicolumn{2}{c|}{\text{sign-invariant}}&16&8&16&2&9&2&4&3&.&2&1\\
E_8(a_7),221&\chi_{1400,20}&\multicolumn{2}{c|}{\text{sign-invariant}}&12&4&8&2&3&.&2&2&.&1&.\\
E_8(a_7),2111&\chi_{70,32}&\multicolumn{2}{c|}{\text{sign-invariant}}&2&.&2&.&1&.&.&.&.&.&.\\
\end{array}
\]
The other block has defect two:
\[
\tiny
\begin{array}{cc|*{3}c}
A_4{+}A_1,2&\chi_{4096,12}&1\\
A_4{+}A_1,11&\chi_{4096,11}&1\\
E_6(a_1){+}A_1,2&\chi_{4096,27}&1\\
E_6(a_1){+}A_1,11&\chi_{4096,26}&1\\
\end{array}
\]
So the image of the $2$-modular Springer correspondence is:
{\tiny
\[
\{0, A_1, 2A_1, 3A_1, 4A_1, (A_2+A_1, 2), A_2+2A_1, A_2+3A_1, A_3 + A_1, A_3+2A_1, (A_4+A_1,2), (A_4+2A_1,2)\}.
\]
}

\subsubsection{Case $\ell = 3$}

The principal block has defect five:
\[
\tiny
\begin{array}{cc|*{22}c}
0&\chi_{1,0}&1&.&.&.&.&.&.&.&.&.&.&.&.&.&.&.&.&.&.&.\\
2A_1&\chi_{35,2}&.&1&.&.&.&.&.&.&.&.&.&.&.&.&.&.&.&.&.&.\\
3A_1&\chi_{84,4}&1&1&1&.&.&.&.&.&.&.&.&.&.&.&.&.&.&.&.&.\\
A_2,11&\chi_{28,8}&.&.&.&1&.&.&.&.&.&.&.&.&.&.&.&.&.&.&.&.\\
4A_1&\chi_{50,8}&1&.&1&.&1&.&.&.&.&.&.&.&.&.&.&.&.&.&.&.\\
A_2{+}A_1,2&\chi_{210,4}&.&1&.&1&.&1&.&.&.&.&.&.&.&.&.&.&.&.&.&.\\
2A_2,2&\chi_{700,6}&.&1&.&.&.&1&1&.&.&.&.&.&.&.&.&.&.&.&.&.\\
2A_2,11&\chi_{300,8}&.&1&1&.&.&1&.&1&.&.&.&.&.&.&.&.&.&.&.&.\\
A_3{+}A_1&\chi_{1344,8}&1&1&1&1&.&1&1&1&1&.&.&.&.&.&.&.&.&.&.&.\\
2A_2{+}2A_1&\chi_{175,12}&.&.&.&.&.&1&.&.&.&1&.&.&.&.&.&.&.&.&.&.\\
D_4&\chi_{525,12}&.&.&.&1&.&.&.&.&1&.&.&.&.&.&.&.&.&.&.&.\\
A_3{+}2A_1&\chi_{1050,10}&.&.&.&.&.&.&1&.&1&.&1&.&.&.&.&.&.&.&.&.\\
D_4(a_1){+}A_1,3&\chi_{1400,8}&.&.&.&1&.&1&.&.&.&1&.&1&.&.&.&.&.&.&.&.\\
D_4(a_1){+}A_1,21&\chi_{1575,10}&.&.&.&2&.&.&.&.&.&.&.&1&1&.&.&.&.&.&.&.\\
D_4(a_1){+}A_1,111&\chi_{350,14}&.&.&.&1&.&.&.&.&.&.&.&.&1&.&.&.&.&.&.&.\\
D_4(a_1){+}A_2,2&\chi_{2240,10}&.&.&.&.&.&1&1&.&.&2&.&1&.&1&.&.&.&.&.&.\\
D_4{+}A_1&\chi_{700,16}&1&.&1&.&1&.&.&1&1&.&1&.&.&.&1&.&.&.&.&.\\
A_4{+}A_1,2&\chi_{4096,12}&1&.&.&1&.&1&1&.&1&1&.&1&1&1&.&1&.&.&.&.\\
2A_3&\chi_{840,14}&.&.&.&.&.&.&1&.&.&.&.&.&.&1&.&.&.&.&.&.\\
D_5(a_1),11&\chi_{2100,16}&.&.&.&1&.&1&.&1&1&.&.&.&1&.&.&1&.&.&.&.\\
A_4{+}2A_1,2&\chi_{4200,12}&.&.&.&1&.&2&1&1&1&2&1&1&.&1&.&1&1&.&.&.\\
A_5&\chi_{3200,16}&1&1&1&.&1&1&1&1&1&1&.&.&.&1&.&1&.&1&.&.\\
D_4{+}A_2,11&\chi_{168,24}&1&.&1&.&1&.&.&1&.&.&.&.&.&.&1&.&.&.&.&.\\
D_5&\chi_{2100,20}&.&.&.&.&.&.&.&1&1&.&.&.&.&.&.&1&.&1&.&.\\
A_4{+}A_3&\chi_{420,20}&.&.&.&1&.&1&.&1&.&1&.&.&.&.&.&.&1&.&.&.\\
E_6(a_3){+}A_1,2&\chi_{3150,18}&.&.&.&2&.&.&.&.&.&2&.&1&1&1&.&.&.&.&1&.\\
D_6(a_2),2&\chi_{4200,18}&.&1&.&1&.&1&1&1&1&1&1&.&1&1&.&1&1&1&.&1\\
D_6(a_2),11&\chi_{2688,20}&1&1&1&1&1&1&.&2&1&1&1&.&.&.&1&1&1&1&.&.\\
E_8(a_7),5&\chi_{4480,16}&.&.&.&2&.&1&.&.&.&2&.&1&1&1&.&1&1&.&1&.\\
E_8(a_7),311&\chi_{1680,22}&.&.&.&.&.&.&.&.&.&.&.&.&1&1&.&1&.&.&.&.\\
D_5{+}A_1&\chi_{3200,22}&1&.&.&1&1&.&.&1&1&.&1&.&1&.&1&1&1&1&.&1\\
E_8(a_7),221&\chi_{1400,20}&.&.&.&.&.&1&.&1&.&.&.&.&.&.&.&1&1&.&.&.\\
E_8(a_7),2111&\chi_{70,32}&.&.&.&.&.&.&.&1&.&.&.&.&.&.&.&.&.&.&.&.\\
E_7(a_4),11&\chi_{700,28}&1&1&1&.&1&.&.&1&.&.&.&.&.&.&1&.&.&1&.&.\\
E_6(a_1),11&\chi_{2100,28}&.&.&.&.&.&.&.&1&.&1&.&.&.&1&.&1&1&1&.&.\\
D_5{+}A_2,11&\chi_{840,26}&.&.&.&.&.&.&.&.&.&.&.&.&1&.&.&.&.&.&.&1\\
E_6&\chi_{525,36}&.&.&.&.&.&.&.&.&.&1&.&.&.&.&.&.&.&1&.&.\\
D_7(a_2),2&\chi_{4200,24}&.&1&.&2&.&1&.&1&.&1&.&.&1&.&.&1&2&1&1&1\\
E_6(a_1){+}A_1,11&\chi_{4096,26}&.&.&.&1&1&.&.&.&.&1&.&.&1&1&.&1&1&1&1&1\\
E_8(b_6),3&\chi_{2240,28}&.&.&.&2&.&.&.&.&.&.&.&.&1&.&.&.&1&.&1&1\\
E_8(b_6),21&\chi_{175,36}&.&.&.&1&.&.&.&.&.&.&.&.&.&.&.&.&1&.&.&.\\
D_7(a_1),11&\chi_{1050,34}&.&1&.&.&.&.&.&.&.&.&.&.&.&.&.&.&.&1&.&1\\
E_8(a_6),21&\chi_{1575,34}&.&.&.&.&.&.&.&.&.&2&.&.&.&1&.&.&.&.&1&.\\
E_7(a_2)&\chi_{1344,38}&.&.&.&.&1&.&.&1&.&1&1&.&.&.&1&.&1&1&.&1\\
E_8(a_6),3&\chi_{1400,32}&.&.&.&1&.&.&.&.&.&1&.&.&.&.&.&.&1&.&1&.\\
E_8(a_6),111&\chi_{350,38}&.&.&.&.&.&.&.&.&.&1&.&.&.&1&.&.&.&.&.&.\\
E_8(a_5),2&\chi_{700,42}&.&.&.&.&.&.&.&.&.&.&1&.&.&.&.&.&1&.&.&1\\
E_8(a_5),11&\chi_{300,44}&.&.&.&.&.&.&.&1&.&.&1&.&.&.&1&.&1&.&.&.\\
E_8(b_4),11&\chi_{50,56}&1&.&.&.&1&.&.&.&.&.&.&.&.&.&1&.&.&.&.&.\\
E_7&\chi_{84,64}&.&.&.&.&1&.&.&.&.&.&1&.&.&.&1&.&.&.&.&.\\
E_8(a_4),2&\chi_{210,52}&.&.&.&.&.&.&.&.&.&1&1&.&.&.&.&.&1&.&.&.\\
E_8(a_3),11&\chi_{28,68}&.&.&.&.&.&.&.&.&.&1&.&.&.&.&.&.&.&.&.&.\\
E_8(a_2)&\chi_{35,74}&.&.&.&.&.&.&.&.&.&.&1&.&.&.&.&.&.&.&.&.\\
E_8&\chi_{1,120}&.&.&.&.&1&.&.&.&.&.&.&.&.&.&.&.&.&.&.&.\\
\end{array}
\]
There are five blocks of defect one, with the following Brauer trees and corresponding pairs:
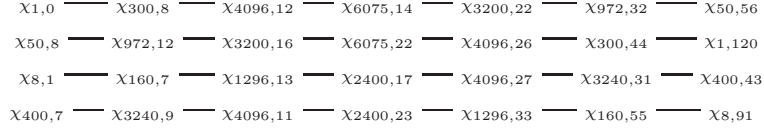
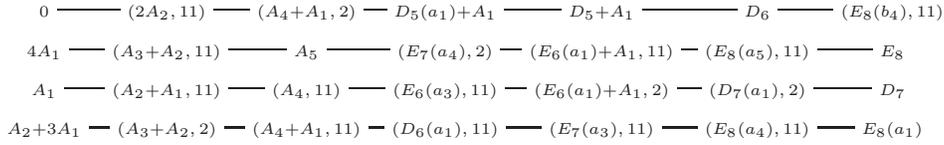
\[
\tiny
\xymatrix@R=.1cm@C=.4cm{
\chi_{567,6}\ar@{-}[r]&\chi_{2835,22}\ar@{-}[r]&\chi_{2268,30}&\qquad&A_3\ar@{-}[r]&A_6{+}A_1\ar@{-}[r]&(E_7(a_3),2)\\
\chi_{2268,10}\ar@{-}[r]&\chi_{2835,14}\ar@{-}[r]&\chi_{567,46}&\qquad&(A_4,2)\ar@{-}[r]&A_4{+}A_2{+}A_1\ar@{-}[r]&E_7(a_1)\\
\chi_{1134,20}\ar@{-}[r]&\chi_{5670,18}\ar@{-}[r]&\chi_{4536,18}&\qquad&(E_6(a_3){+}A_1,11)\ar@{-}[r]&(E_8(a_7),41)\ar@{-}[r]&(E_8(a_7),32)\\
\chi_{1296,13}\ar@{-}[r]&\chi_{4536,23}\ar@{-}[r]&\chi_{3240,31}&\qquad&(A_4,11)\ar@{-}[r]&(D_5{+}A_2,2)\ar@{-}[r]&(D_7(a_1),2)\\
\chi_{3240,9}\ar@{-}[r]&\chi_{4536,13}\ar@{-}[r]&\chi_{1296,33}&\qquad&(A_3{+}A_2,2)\ar@{-}[r]&A_4{+}A_2\ar@{-}[r]&(E_7(a_3),11)\\
}
\]
Four characters form blocks of defect zero:
\[
\tiny
\begin{array}{*{4}c}
\chi_{972,12}&\chi_{6075,14}&\chi_{6075,22}&\chi_{972,32}\\
(A_3{+}A_2,11)&D_5(a_1){+}A_1&(E_7(a_4),2)&D_6\\
\end{array}
\]

Finally, there is another block of full defect:
\[
\tiny
\begin{array}{cc|*{15}c}
A_1&\chi_{8,1}&1&.&.&.&.&.&.&.&.&.&.&.&.\\
A_2,2&\chi_{112,3}&1&1&.&.&.&.&.&.&.&.&.&.&.\\
A_2{+}A_1,11&\chi_{160,7}&.&1&1&.&.&.&.&.&.&.&.&.&.\\
A_2{+}2A_1&\chi_{560,5}&2&1&1&1&.&.&.&.&.&.&.&.&.\\
A_2{+}3A_1&\chi_{400,7}&1&.&.&1&1&.&.&.&.&.&.&.&.\\
2A_2{+}A_1&\chi_{448,9}&1&.&.&1&.&1&.&.&.&.&.&.&.\\
D_4(a_1),3&\chi_{1400,7}&1&1&1&1&.&1&1&.&.&.&.&.&.\\
D_4(a_1),21&\chi_{1008,9}&.&1&2&.&.&.&1&.&.&.&.&.&.\\
D_4(a_1),111&\chi_{56,19}&.&.&1&.&.&.&.&.&.&.&.&.&.\\
A_3{+}A_2{+}A_1&\chi_{1400,11}&.&.&.&1&1&2&1&1&.&.&.&.&.\\
D_4(a_1){+}A_2,11&\chi_{840,13}&1&.&.&1&.&.&.&.&1&.&.&.&.\\
A_4{+}A_1,11&\chi_{4096,11}&2&.&1&2&1&2&1&.&1&1&.&.&.\\
D_5(a_1),2&\chi_{2800,13}&.&.&1&.&.&1&1&.&.&1&.&.&.\\
A_4{+}2A_1,11&\chi_{3360,13}&1&.&1&1&1&2&1&1&.&1&.&.&.\\
D_4{+}A_2,2&\chi_{4200,15}&1&.&1&1&2&2&1&1&1&1&1&.&.\\
E_6(a_3),2&\chi_{5600,15}&1&.&1&1&1&2&1&.&1&1&.&1&.\\
E_6(a_3),11&\chi_{2400,17}&.&.&1&.&.&.&.&.&1&1&.&.&.\\
A_5{+}A_1&\chi_{2016,19}&.&1&2&.&.&2&1&1&.&.&.&.&1\\
D_5(a_1){+}A_2&\chi_{1344,19}&1&.&1&1&1&1&.&.&1&.&1&.&.\\
E_7(a_5),3&\chi_{7168,17}&2&1&3&1&2&3&1&1&1&1&1&1&1\\
E_7(a_5),21&\chi_{5600,19}&2&.&1&1&2&1&.&.&2&1&1&1&.\\
E_7(a_5),111&\chi_{448,25}&.&.&.&.&.&.&.&.&1&.&.&.&.\\
D_6(a_1),2&\chi_{5600,21}&1&.&2&.&1&1&.&.&1&1&1&1&1\\
A_6&\chi_{4200,21}&2&1&2&1&1&1&.&.&1&.&1&1&1\\
D_6(a_1),11&\chi_{2400,23}&.&.&.&.&.&1&.&.&1&.&.&1&.\\
E_6(a_1),2&\chi_{2800,25}&.&.&1&.&.&1&.&.&.&.&.&1&1\\
D_7(a_2),11&\chi_{3360,25}&1&1&2&.&1&1&.&.&.&.&1&1&1\\
E_6(a_1){+}A_1,2&\chi_{4096,27}&1&.&2&.&2&1&.&.&1&.&2&1&1\\
A_7&\chi_{1400,29}&1&1&2&.&.&.&.&.&.&.&1&.&1\\
E_8(b_6),111&\chi_{840,31}&.&.&.&.&1&.&.&.&1&.&1&.&.\\
E_6{+}A_1&\chi_{448,39}&.&.&1&.&1&.&.&.&.&.&1&.&.\\
E_8(b_5),3&\chi_{1400,37}&.&.&1&.&1&1&.&1&.&.&1&.&1\\
E_8(b_5),21&\chi_{1008,39}&.&.&.&.&.&2&.&1&.&.&.&.&1\\
E_8(b_5),111&\chi_{56,49}&.&.&.&.&.&1&.&.&.&.&.&.&.\\
D_7&\chi_{400,43}&1&.&.&.&1&.&.&.&.&.&1&.&.\\
E_8(b_4),2&\chi_{560,47}&.&.&.&.&2&1&.&1&.&.&1&.&.\\
E_8(a_4),11&\chi_{160,55}&.&.&.&.&.&1&.&1&.&.&.&.&.\\
E_8(a_3),2&\chi_{112,63}&.&.&.&.&1&.&.&1&.&.&.&.&.\\
E_8(a_1)&\chi_{8,91}&.&.&.&.&1&.&.&.&.&.&.&.&.\\
\end{array}
\]

\subsubsection{Case $\ell = 5$}

There are 5 blocks of defect one, with the following Brauer trees:
\[
\tiny
\xymatrix@R=.1cm@C=.4cm{
\chi_{35,2}\ar@{-}[r]&\chi_{840,14}\ar@{-}[r]&\chi_{2835,22}\ar@{-}[r]&\chi_{2240,28}\ar@{-}[r]&\chi_{210,52}\\
\chi_{210,4}\ar@{-}[r]&\chi_{2240,10}\ar@{-}[r]&\chi_{2835,14}\ar@{-}[r]&\chi_{840,26}\ar@{-}[r]&\chi_{35,74}\\
\chi_{420,20}\ar@{-}[r]&\chi_{4480,16}\ar@{-}[r]&\chi_{5670,18}\ar@{-}[r]&\chi_{1680,22}\ar@{-}[r]&\chi_{70,32}\\
\chi_{160,7}\ar@{-}[r]&\chi_{840,13}\ar@{-}[r]&\chi_{3360,25}\ar@{-}[r]&\chi_{3240,31}\ar@{-}[r]&\chi_{560,47}\\
\chi_{560,5}\ar@{-}[r]&\chi_{3240,9}\ar@{-}[r]&\chi_{3360,13}\ar@{-}[r]&\chi_{840,31}\ar@{-}[r]&\chi_{160,55}\\
}
\]
and corresponding pairs
\[
\tiny
\xymatrix@R=.1cm@C=.4cm{
2A_1\ar@{-}[r]&2A_3\ar@{-}[r]&A_6{+}A_1\ar@{-}[r]&(E_8(b_6),3)\ar@{-}[r]&(E_8(a_4),2)\\
(A_2{+}A_1,2)\ar@{-}[r]&(D_4(a_1){+}A_2,2)\ar@{-}[r]&A_4{+}A_2{+}A_1\ar@{-}[r]&(D_5{+}A_2,11)\ar@{-}[r]&E_8(a_2)\\
A_4{+}A_3\ar@{-}[r]&(E_8(a_7),5)\ar@{-}[r]&(E_8(a_7),41)\ar@{-}[r]&(E_8(a_7),311)\ar@{-}[r]&(E_8(a_7),2111)\\
(A_2{+}A_1,11)\ar@{-}[r]&(D_4(a_1){+}A_2,11)\ar@{-}[r]&(D_7(a_2),11)\ar@{-}[r]&(D_7(a_1),2)\ar@{-}[r]&(E_8(b_4),2)\\
A_2{+}2A_1\ar@{-}[r]&(A_3{+}A_2,2)\ar@{-}[r]&(A_4{+}2A_1,11)\ar@{-}[r]&(E_8(b_6),111)\ar@{-}[r]&(E_8(a_4),11)\\
}
\]
The other 47 characters have defect zero:
\[
 \tiny
\begin{array}{*{47}c}
\chi_{50,8}&\chi_{400,7}&\chi_{700,6}&\chi_{300,8}&\chi_{1400,7}&\chi_{175,12}&\chi_{525,12}&\chi_{1050,10}&\chi_{1400,8}\\
4A_1&A_2{+}3A_1&(2A_2,2)&(2A_2,11)&(D_4(a_1),3)&2A_2{+}2A_1&D_4&A_3{+}2A_1&(D_4(a_1){+}A_1,3)\\
\end{array}
\]
\[
\tiny
\begin{array}{*{47}c}
\chi_{1575,10}&\chi_{350,14}&\chi_{1400,11}&\chi_{700,16}&\chi_{2800,13}&\chi_{2100,16}&\chi_{4200,12}\\
(D_4(a_1){+}A_1,21)&(D_4(a_1){+}A_1,111)&A_3{+}A_2{+}A_1&D_4{+}A_1&(D_5(a_1),2)&(D_5(a_1),11)&(A_4{+}2A_1,2)
\end{array}
\]
\[
\tiny
\begin{array}{*{47}c}
\chi_{3200,16}&\chi_{6075,14}&\chi_{4200,15}&\chi_{5600,15}&\chi_{2400,17}&\chi_{2100,20}&\chi_{3150,18}&\chi_{4200,18}\\
A_5&D_5(a_1){+}A_1&(D_4{+}A_2,2)&(E_6(a_3),2)&(E_6(a_3),11)&D_5&(E_6(a_3){+}A_1,2)&(D_6(a_2),2)\\
\end{array}
\]
\[
\tiny
\begin{array}{*{47}c}
\chi_{5600,19}&\chi_{3200,22}&\chi_{1400,20}&\chi_{5600,21}&\chi_{4200,21}&\chi_{2400,23}&\chi_{6075,22}&\chi_{700,28}\\
(E_7(a_5),21)&D_5{+}A_1&(E_8(a_7),221)&(D_6(a_1),2)&A_6&(D_6(a_1),11)&(E_7(a_4),2)&(E_7(a_4),11)\\
\end{array}
\]
\[
\tiny
\begin{array}{*{47}c}
\chi_{2800,25}&\chi_{2100,28}&\chi_{525,36}&\chi_{4200,24}&\chi_{1400,29}&\chi_{175,36}&\chi_{1050,34}&\chi_{1575,34}\\
(E_6(a_1),2)&(E_6(a_1),11)&E_6&(D_7(a_2),2)&A_7&(E_8(b_6),21)&(D_7(a_1),11)&(E_8(a_6),21)\\
\end{array}
\]
\[
\tiny
\begin{array}{*{47}c}
\chi_{1400,32}&\chi_{350,38}&\chi_{1400,37}&\chi_{400,43}&\chi_{700,42}&\chi_{300,44}&\chi_{50,56}\\
(E_8(a_6),3)&(E_8(a_6),111)&(E_8(b_5),3)&D_7&(E_8(a_5),2)&(E_8(a_5),11)&(E_8(b_4),11)\\
\end{array}
\]

\subsubsection{Case $\ell = 7$}
There are four blocks of defect one, with the following Brauer trees:
\[
\tiny
\xymatrix@R=.1cm@C=.4cm{
\chi_{1,0}\ar@{-}[r]&\chi_{300,8}\ar@{-}[r]&\chi_{4096,12}\ar@{-}[r]&\chi_{6075,14}\ar@{-}[r]&\chi_{3200,22}\ar@{-}[r]&\chi_{972,32}\ar@{-}[r]&\chi_{50,56}\\
\chi_{50,8}\ar@{-}[r]&\chi_{972,12}\ar@{-}[r]&\chi_{3200,16}\ar@{-}[r]&\chi_{6075,22}\ar@{-}[r]&\chi_{4096,26}\ar@{-}[r]&\chi_{300,44}\ar@{-}[r]&\chi_{1,120}\\
\chi_{8,1}\ar@{-}[r]&\chi_{160,7}\ar@{-}[r]&\chi_{1296,13}\ar@{-}[r]&\chi_{2400,17}\ar@{-}[r]&\chi_{4096,27}\ar@{-}[r]&\chi_{3240,31}\ar@{-}[r]&\chi_{400,43}\\
\chi_{400,7}\ar@{-}[r]&\chi_{3240,9}\ar@{-}[r]&\chi_{4096,11}\ar@{-}[r]&\chi_{2400,23}\ar@{-}[r]&\chi_{1296,33}\ar@{-}[r]&\chi_{160,55}\ar@{-}[r]&\chi_{8,91}\\
}
\]
and corresponding pairs
\[
\tiny
\xymatrix@R=.1cm@C=.2cm{
0\ar@{-}[r]&(2A_2,11)\ar@{-}[r]&(A_4{+}A_1,2)\ar@{-}[r]&D_5(a_1){+}A_1\ar@{-}[r]&D_5{+}A_1\ar@{-}[r]&D_6\ar@{-}[r]&(E_8(b_4),11)\\
4A_1\ar@{-}[r]&(A_3{+}A_2,11)\ar@{-}[r]&A_5\ar@{-}[r]&(E_7(a_4),2)\ar@{-}[r]&(E_6(a_1){+}A_1,11)\ar@{-}[r]&(E_8(a_5),11)\ar@{-}[r]&E_8\\
A_1\ar@{-}[r]&(A_2{+}A_1,11)\ar@{-}[r]&(A_4,11)\ar@{-}[r]&(E_6(a_3),11)\ar@{-}[r]&(E_6(a_1){+}A_1,2)\ar@{-}[r]&(D_7(a_1),2)\ar@{-}[r]&D_7\\
A_2{+}3A_1\ar@{-}[r]&(A_3{+}A_2,2)\ar@{-}[r]&(A_4{+}A_1,11)\ar@{-}[r]&(D_6(a_1),11)\ar@{-}[r]&(E_7(a_3),11)\ar@{-}[r]&(E_8(a_4),11)\ar@{-}[r]&E_8(a_1)\\
}
\]
There are 84 blocks of defect zero:
\[
 \tiny
\begin{array}{cc|cc|cc}
\chi_{35,2}&2A_1&\chi_{3360,13}&(A_4{+}2A_1,11)&\chi_{2100,28}&(E_6(a_1),11)\\
\chi_{84,4}&3A_1&\chi_{4536,13}&A_4{+}A_2&\chi_{4536,23}&(D_5{+}A_2,2)\\
\chi_{112,3}&(A_2,2)&\chi_{2835,14}&A_4{+}A_2{+}A_1&\chi_{840,26}&(D_5{+}A_2,11)\\
\chi_{28,8}&(A_2,11)&\chi_{4200,15}&(D_4{+}A_2,2)&\chi_{525,36}&E_6\\
\chi_{210,4}&(A_2{+}A_1,2)&\chi_{168,24}&(D_4{+}A_2,11)&\chi_{4200,24}&(D_7(a_2),2)\\
\chi_{560,5}&A_2{+}2A_1&\chi_{5600,15}&(E_6(a_3),2)&\chi_{3360,25}&(D_7(a_2),11)\\
\chi_{567,6}&A_3&\chi_{2100,20}&D_5&\chi_{1400,29}&A_7\\
\chi_{700,6}&(2A_2,2)&\chi_{420,20}&A_4{+}A_3&\chi_{2268,30}&(E_7(a_3),2)\\
\chi_{448,9}&2A_2{+}A_1&\chi_{2016,19}&A_5{+}A_1&\chi_{2240,28}&(E_8(b_6),3)\\
\chi_{1344,8}&A_3{+}A_1&\chi_{1344,19}&D_5(a_1){+}A_2&\chi_{175,36}&(E_8(b_6),21)\\
\chi_{1400,7}&(D_4(a_1),3)&\chi_{3150,18}&(E_6(a_3){+}A_1,2)&\chi_{840,31}&(E_8(b_6),111)\\
\chi_{1008,9}&(D_4(a_1),21)&\chi_{1134,20}&(E_6(a_3){+}A_1,11)&\chi_{448,39}&E_6{+}A_1\\
\chi_{56,19}&(D_4(a_1),111)&\chi_{4200,18}&(D_6(a_2),2)&\chi_{1050,34}&(D_7(a_1),11)\\
\chi_{175,12}&2A_2{+}2A_1&\chi_{2688,20}&(D_6(a_2),11)&\chi_{1575,34}&(E_8(a_6),21)\\
\chi_{525,12}&D_4&\chi_{7168,17}&(E_7(a_5),3)&\chi_{1344,38}&E_7(a_2)\\
\chi_{1050,10}&A_3{+}2A_1&\chi_{5600,19}&(E_7(a_5),21)&\chi_{1400,32}&(E_8(a_6),3)\\
\chi_{1400,8}&(D_4(a_1){+}A_1,3)&\chi_{448,25}&(E_7(a_5),111)&\chi_{350,38}&(E_8(a_6),111)\\
\chi_{1575,10}&(D_4(a_1){+}A_1,21)&\chi_{4480,16}&(E_8(a_7),5)&\chi_{1400,37}&(E_8(b_5),3)\\
\chi_{350,14}&(D_4(a_1){+}A_1,111)&\chi_{5670,18}&(E_8(a_7),41)&\chi_{1008,39}&(E_8(b_5),21)\\
\chi_{2268,10}&(A_4,2)&\chi_{1680,22}&(E_8(a_7),311)&\chi_{56,49}&(E_8(b_5),111)\\
\chi_{1400,11}&A_3{+}A_2{+}A_1&\chi_{4536,18}&(E_8(a_7),32)&\chi_{700,42}&(E_8(a_5),2)\\
\chi_{2240,10}&(D_4(a_1){+}A_2,2)&\chi_{1400,20}&(E_8(a_7),221)&\chi_{567,46}&E_7(a_1)\\
\chi_{840,13}&(D_4(a_1){+}A_2,11)&\chi_{70,32}&(E_8(a_7),2111)&\chi_{560,47}&(E_8(b_4),2)\\
\chi_{700,16}&D_4{+}A_1&\chi_{5600,21}&(D_6(a_1),2)&\chi_{84,64}&E_7\\
\chi_{840,14}&2A_3&\chi_{4200,21}&A_6&\chi_{210,52}&(E_8(a_4),2)\\
\chi_{2800,13}&(D_5(a_1),2)&\chi_{700,28}&(E_7(a_4),11)&\chi_{112,63}&(E_8(a_3),2)\\
\chi_{2100,16}&(D_5(a_1),11)&\chi_{2835,22}&A_6{+}A_1&\chi_{28,68}&(E_8(a_3),11)\\
\chi_{4200,12}&(A_4{+}2A_1,2)&\chi_{2800,25}&(E_6(a_1),2)&\chi_{35,74}&E_8(a_2)\\
\end{array}
\]

The pairs $E_8$, $E_8(a_1)$, $(E_8(b_4),11)$, $D_7$ and $(E_8(a_7), 11111)$ are missing, the latter being
the modular reduction of the characteristic zero cuspidal pair.

\def\cprime{$'$}

\end{document}